\documentclass[11pt]{article}
\usepackage{amssymb}
\usepackage{amsfonts}
\usepackage{amsmath}
\usepackage{tikz-cd}
\usepackage{tikz}
\usepackage{pgfplots}
\usepackage{faktor}
\usepackage{graphicx,nicefrac}
\usepackage[nottoc]{tocbibind}
\usepackage{microtype}
\usepackage{capt-of}
\usepackage{hyperref}
\usepackage{doi}

\pgfplotsset{compat=newest}

\newtheorem{theorem}{Theorem}[section]

\newtheorem{corollary}[theorem]{Corollary}

\newtheorem{definition}[theorem]{Definition}

\newtheorem{lemma}[theorem]{Lemma}

\newtheorem{proposition}[theorem]{Proposition}
\newtheorem{remark}[theorem]{Remark}

\numberwithin{equation}{section}
\newenvironment{proof}[1][Proof]{\noindent\textbf{#1.} }{\ \rule{0.5em}{0.5em}}
\def\func#1{\mathop{\rm #1}}  

\newcommand{\R}{\mathbb{R}}
\newcommand{\C}{\mathbb{C}}
\renewcommand{\O}{\mathbb{O}}

\newcommand{\SU}{\mathrm{SU}}
\newcommand{\GL}{\mathrm{GL}}
\newcommand{\Gtwo}{G_2}

\newcommand{\id}{\mathrm{id}}
\newcommand{\Aut}{\mathrm{Aut}}
\newcommand{\End}{\mathrm{End}}

\newcommand{\Span}{\mathrm{Span}}

\newcommand{\Spec}{\mathrm{Spec}}
\newcommand{\tr}{\mathrm{tr}}
\newcommand{\ImO}{\mathrm{Im}\mathbb{O}}

\renewcommand{\Im}{\mathrm{Im}}

\newcommand{\tens}{\otimes}
\newcommand{\tensor}{\otimes}

\newcommand{\ip}[2]{\left\langle #1,#2\right\rangle}

\newcommand{\one}{\mathbf{1}}
\newcommand{\seven}{\mathbf{7}}
\newcommand{\twelve}{\mathbf{12}}
\newcommand{\fourteen}{\mathbf{14}}
\newcommand{\twentyseven}{\mathbf{27}}
\newcommand{\eight}{\mathbf{8}}
\newcommand{\six}{\mathbf{6}}
\newcommand{\three}{\mathbf{3}}
\newcommand{\threebar}{\overline{\mathbf{3}}}

\newcommand{\Vhat}{\widehat{V}}
\newcommand{\ehat}[1]{\hat{e}_{#1}}
\newcommand{\what}{\hat{w}}
\newcommand{\zhat}{\hat{z}}
\newcommand{\vhat}{\hat{v}}
\newcommand{\uhat}{\hat{u}}

\title{Octonionic structure operator and its right spectrum}
\author{Sergey Grigorian \\
School of Mathematical \& Statistical Sciences\\
University of Texas Rio Grande Valley\\
Edinburg, TX 78539\\
USA}       
\date{June 6, 2026}

\begin{document}
\maketitle

\begin{abstract}
We study a canonical \(G_2\)-equivariant operator
\[
h:\mathbb O\otimes_{\mathbb R}V\longrightarrow \mathbb O\otimes_{\mathbb R}V
\]
defined using only octonion multiplication, where \(V\) is the standard \(7\)-dimensional \(G_2\)-module. We first compute its ordinary real spectrum using the \(G_2\)-decomposition of \(\mathbb O\otimes_{\mathbb R}V\). We then analyze the octonionic right-eigenvalue problem \(h(\what)=\what\lambda\) for $\lambda \in \O$. After fixing a complex slice \(\mathbb R\oplus\mathbb R\uhat\subset\mathbb O\), the problem becomes a real spectral problem for \(H_{u,Q}=h-QR_{\uhat}\), whose residual symmetry is \(\SU(3)\). The resulting \(\SU(3)\)-block decomposition yields two explicit spectral loci in each slice: a quartic curve and a circle. The equations that define the loci are independent of the slice and the full spectrum is obtained by allowing $\uhat$ to vary.
\end{abstract}


\section{Introduction}

The octonions \(\mathbb{O}\) occupy a distinctive place in geometry and algebra. They provide the linear algebra underlying \(G_2\) and \(\mathrm{Spin}(7)\) structures, furnish the \(7\)-dimensional cross product, and appear in a range of exceptional constructions from calibrated geometry to Jordan-theoretic models of exceptional symmetry \cite{Baez2002Octonions,BrownGray1967VectorCrossProducts,Bryant1987ExceptionalHolonomy,HarveyLawson1982Calibrated,SalamonWalpuski2017NotesOctonions,SpringerVeldkamp2000OctonionsJordanExceptional}. At the same time, linear algebra over \(\mathbb{O}\) is markedly subtler than in the associative setting. Questions that in the real, complex, or quaternionic cases are routine, such as the formulation of eigenvalue problems, orthogonality of eigenspaces, or the relation between Hermitian structure and spectral data, become substantially more delicate over \(\mathbb{O}\). This has led to a small but important literature on octonionic matrices and their eigenvalue theories, particularly for Hermitian \(2\times 2\) and \(3\times 3\) matrices and for exceptional Jordan-algebraic variants \cite{DrayJaneskyManogue2000NonRealEigenvalues,DrayManogue1998OctonionicEigenvalue,DrayManogue1999ExceptionalJordan,DrayManogueOkubo2002OrthonormalEigenbases,GillowWilesDray2010ImaginaryEigenvalues,DeLeoDucati2012OctonionicEigenvalue,Okubo1999Symmetric3x3Octonionic}.

The purpose of this paper is to study a different but natural octonionic spectral problem, arising not from a small octonionic matrix algebra but from the \(G_2\)-representation
\begin{equation}\label{eq:intro-W}
W=\mathbb{O}\otimes_{\mathbb{R}}V,
\end{equation}
where \(V\) is a Euclidean \(7\)-space identified with $\ImO$ as a $G_2$-module. On \(W\) there is a canonically defined \(G_2\)-equivariant operator
\[
h:W\to W
\]
obtained from the octonion multiplication tensor. In this sense \(h\) may be viewed as an octonionic structure operator: it is built from the same algebraic data that determines octonion multiplication, but acts on the octonionic extension \(W\) rather than on \(\O\) itself.

A central reason for focusing on \emph{right} eigenvalues is that this formulation makes the role of nonassociativity especially clear. In the quaternionic setting, the natural framework is that of right \(\mathbb H\)-modules and right-linear operators. In this associative but noncommutative context, Hermitian (self-adjoint) operators still satisfy a largely classical spectral theory: in particular, Hermitian quaternionic matrices have real right eigenvalues, and the corresponding spectral theorem closely parallels the complex case \cite{AlpayColomboKimsey2016Spectral,Baker1999RightEigenvalues,FarenickPidkowich2003Spectral,Zhang1997QuaternionsMatrices}. Thus noncommutativity alone does not force the appearance of non-real right-spectrum for Hermitian operators. By formulating the octonionic problem in the directly analogous right-linear manner, we make the contrast with the quaternionic case as sharp as possible. The resulting behavior of the right spectrum should therefore be understood as a genuinely nonassociative spectral phenomenon, rather than as a byproduct of noncommutativity alone.

Our first goal is to determine the ordinary real spectrum of \(h\), which may be regarded as a self-adjoint real-linear operator. Since \(h\) is \(G_2\)-equivariant, its action is constrained by the decomposition of \(W\) into irreducible \(G_2\)-summands. This leads to a finite-dimensional representation-theoretic calculation and yields a complete description of the real eigenvalues of \(h\). Our second goal is to study the octonionic right-eigenvalue problem
\begin{equation}\label{eq:intro-right-eigenvalue}
h(\vhat)=\vhat\lambda,\qquad \lambda\in \mathbb{O}, \vhat \in W,
\end{equation}
which is more subtle. One should not expect a field-based spectral formalism analogous to the complex case. Rather, the equation \(h(\vhat)=\vhat\lambda\) expresses the compatibility between two canonical structures on \(W\): the operator \(h\) itself and right multiplication by \(\lambda\). Its interest lies in the structure of the resulting alignment locus
\begin{equation}\label{eq:intro-alignment-locus}
\{\lambda\in \mathbb{O}: \ker(h-R_\lambda)\neq 0\},
\end{equation}
and in how this locus reflects the exceptional algebraic geometry of \(\mathbb O\). In contrast with the associative cases, right-eigenvalues over \(\mathbb{O}\) need not reduce to a discrete real spectral set, and continuous families of non-real eigenvalues can arise even for Hermitian operators in the appropriate sense \cite{DrayJaneskyManogue2000NonRealEigenvalues,DrayManogue1998OctonionicEigenvalue,GillowWilesDray2010ImaginaryEigenvalues}. Continuous non-real octonionic eigenvalue families are therefore not unprecedented; what is new here is that they arise for a canonical \(G_2\)-equivariant operator and admit a symmetry-driven description.

More concretely, after fixing a unit imaginary octonion \(\uhat\in \operatorname{Im}\mathbb{O}\), we analyze right-eigenvalues lying in the complex line
\begin{equation}\label{eq:intro-complex-slice}
\mathbb{R}\oplus \mathbb{R}\uhat \subset \mathbb{O}.
\end{equation}
Writing \(\lambda=\mu+Q\uhat\), for \(Q\in\R\), the right-eigenvalue equation can be rewritten as an ordinary real eigenvalue problem for a deformed operator
\begin{equation}\label{eq:intro-HuQ}
H_{u,Q}=h-Q\,R_{\uhat},
\end{equation}
where \(R_{\uhat}\) denotes right multiplication by \(\uhat\). The point is that \(H_{u,Q}\) is no longer \(G_2\)-equivariant, but it remains equivariant under the stabilizer of \(\uhat\), namely \(\mathrm{SU}(3)\subset G_2\). This reduces the octonionic spectral problem to a blockwise analysis on the \(\mathrm{SU}(3)\)-isotypic decomposition of \(W\). In this way, the non-real spectral geometry of \(h\) becomes accessible by a combination of representation theory and explicit matrix calculation.

The resulting picture is structurally different from the standard spectral theory of self-adjoint operators over \(\mathbb{R}\) or \(\mathbb{C}\). On the one hand, the real spectrum of \(h\) is discrete and determined by \(G_2\)-equivariance. On the other hand, the octonionic right-spectrum contains continuous families inside the complex slices \(\mathbb{R}\oplus \mathbb{R}\uhat\), and these families can be described explicitly in terms of low-dimensional \(\mathrm{SU}(3)\)-equivariant blocks. Thus the operator \(h\) provides a concrete example in which exceptional symmetry both constrains and reveals genuinely nonclassical spectral behavior.

This reduction may also be viewed through the lens of operator pencils.  For a fixed unit imaginary direction \(\uhat\), the right-eigenvalue equation is equivalent to the condition that the two-parameter real-linear pencil
\begin{equation}\label{eq:intro-pencil}
h-\mu\id-Q R_{\uhat}
\end{equation}
has non-trivial kernel. Thus, rather than asking for the roots of a one-parameter characteristic polynomial, the octonionic right-spectrum is described by determinant loci in the \((\mu,Q)\)-plane.  The \(\SU(3)\)-equivariant decomposition of \(W\) makes this pencil block diagonal, and the spectral curves obtained below arise as the corresponding block determinant equations.  This is analogous in spirit to generalized eigenvalue and matrix-pencil problems \cite{GohbergLancasterRodman1982MatrixPolynomials,TisseurMeerbergen2001QuadraticEigenvalue}, but with an exceptional feature: the pencil is forced by octonionic right multiplication and its symmetry reduction \(G_2\to \SU(3)\).

The paper makes the following contributions.
\begin{enumerate}
    \item We define a natural \(G_2\)-equivariant operator \(h\) on the octonionic extension \(W=\mathbb{O}\otimes_{\mathbb{R}}V\), characterized intrinsically as an operator of left octonionic multiplication type induced by the octonion product.
    \item We determine the real spectrum of \(h\) by decomposing \(W\) into irreducible \(G_2\)-representations and computing the scalar or block action of \(h\) on each summand.
    \item We reformulate the octonionic right-eigenvalue problem for \(h\) as a family of \(\mathrm{SU}(3)\)-equivariant real spectral problems \(H_{u,Q}\what=\mu \what\) for $\what \in W$. 
    \item We analyze these \(\mathrm{SU}(3)\)-equivariant blocks explicitly and show that, in each complex slice, the non-real right-spectrum is governed by a quartic locus from the \(\one^{\oplus4}\)-block and a circle from the \(\eight^{\oplus2}\)-block.
    \item We interpret the resulting spectral loci as a symmetry-controlled instance of genuinely octonionic spectral behavior, distinct from but related to the earlier Hermitian-matrix literature.
\end{enumerate}

From the perspective of existing work, our emphasis differs from the classical studies of small octonionic Hermitian matrices. Those works focus on matrix models, determinant identities, orthonormal eigenbases, and exceptional Jordan analogues \cite{DrayManogue1998OctonionicEigenvalue,DrayManogue1999ExceptionalJordan,DrayJaneskyManogue2000NonRealEigenvalues,DrayManogueOkubo2002OrthonormalEigenbases,Okubo1999Symmetric3x3Octonionic}. By contrast, the present paper begins with a canonical operator attached to the \(G_2\)-module \(\mathbb{O}\otimes_{\mathbb{R}}V\) and exploits representation theory as the primary organizing principle. In particular, the continuous non-real spectral families that appear here are not inserted ad hoc through a chosen matrix model; rather, they emerge from the interaction between octonionic right multiplication and the branching from \(G_2\) symmetry to \(\mathrm{SU}(3)\) symmetry after fixing a complex slice.

This point of view also suggests a geometric interpretation. Since the linear algebra of \(\operatorname{Im}\mathbb{O}\) underlies \(G_2\)-geometry, one may regard \(h\) as the pointwise model for a natural operator on the octonion bundle of a \(G_2\)-manifold, or more generally on octonion-valued tensor fields. We do not pursue such manifold-level constructions here beyond brief remarks, but the algebraic results obtained below indicate that the spectral behavior of such operators may encode geometric information not visible in ordinary real or complex formulations \cite{Bryant1987ExceptionalHolonomy,Grigorian2017G2StructuresOctonionBundles,Karigiannis2019IntroductionG2}.

The paper is organized as follows. In Section~\ref{sec:background} we review the necessary background on octonions, the group \(G_2\), and the decomposition of the \(G_2\)-module \(W=\mathbb{O}\otimes_{\mathbb{R}}V\). In Section~\ref{sec:operator-h} we define the operator \(h\) and establish its basic algebraic properties. Section~\ref{sec:real-spectrum} computes the real spectrum of \(h\) using the \(G_2\)-irreducible decomposition of \(W\). In Section~\ref{sec:right-spectrum-reduction} we reformulate the octonionic right-eigenvalue problem in terms of the operators \(H_{u,Q}\) and the \(\mathrm{SU}(3)\)-decomposition associated to a fixed unit \(\uhat\in \operatorname{Im}\mathbb{O}\). Section~\ref{sec:block-analysis} contains the blockwise analysis of these operators, and Section~\ref{sec:main-right-spectrum} assembles the results into the description of the octonionic right-spectrum of \(h\). We conclude with remarks on the geometry of the spectral loci and possible extensions to \(G_2\)-manifolds.

\noindent\textbf{Acknowledgements.} I thank Spiro Karigiannis for multiple discussions on the \(h\) operator and also Jonathan D.H. Smith for his helpful comments. AI tools were used for editing and formatting assistance. Calculations were performed using \emph{Maple}.

\section{Background on octonions, $G_2$, and the octonionic extension}\label{sec:background}

In this section we recall the basic algebraic and representation-theoretic facts needed later. Standard references include \cite{Baez2002Octonions,BrownGray1967VectorCrossProducts,Karigiannis2019IntroductionG2,SalamonWalpuski2017NotesOctonions,SpringerVeldkamp2000OctonionsJordanExceptional}.

\subsection{Octonions and the cross product}

Let \(\mathbb{O}\) denote the algebra of octonions, equipped with its standard conjugation
\[
x \mapsto \bar x,
\]
and norm
\[
|x|^2 = x\bar x = \bar x x.
\]
We write
\[
\operatorname{Re}\mathbb{O} \cong \mathbb{R},
\qquad
\operatorname{Im}\mathbb{O}
=
\{x\in \mathbb{O} : \bar x = -x\},
\]
so that
\[
\mathbb{O} = \mathbb{R}\oplus \operatorname{Im}\mathbb{O}.
\]
The Euclidean inner product on \(\mathbb{O}\) is given by
\[
\langle x,y\rangle = \operatorname{Re}(x\bar y),
\]
and restricts to a Euclidean metric \(g\) on \(\operatorname{Im}\mathbb{O}\), which is a real vector space of dimension \(7\).

For \(u,v\in \ImO\), the octonion product decomposes as
\begin{equation}\label{eq:octonion-product-imaginary}
uv = -g(u,v) + u\times v,
\end{equation}
where \(u\times v\in \ImO\) defines the \(7\)-dimensional cross product. Equivalently, the cross product is characterized by
\begin{equation}\label{eq:cross-product-phi}
g(u\times v,w)=\varphi(u,v,w),
\end{equation}
where \(\varphi\in \Lambda^3 (\ImO)^*\) is the positive \(3\)-form associated to the octonion algebra structure. In particular,
\[
u\times v = -\,v\times u,
\qquad
g(u\times v,u)=g(u\times v,v)=0,
\]
and
\[
|u\times v|^2 = |u|^2|v|^2-g(u,v)^2.
\]

The data \((g,\varphi)\) determine the octonion multiplication on \(\mathbb{R}\oplus \ImO\) up to the usual overall sign convention. Conversely, a choice of positive \(3\)-form \(\varphi\) on an oriented Euclidean \(7\)-space $\Vhat$ determines the corresponding cross product and hence the octonion multiplication law on \(\mathbb{R}\oplus \Vhat\).
Given an orthonormal basis \(\{\ehat{1},...,\ehat{7}\}\) on \(\Vhat\), multiplication in components becomes: \begin{equation}\label{eq:octonion-multiplication-components}
    \ehat{i} \ehat{j} = -g_{ij}1+\varphi_{ij}{}^k \ehat{k},
\end{equation}
where we are using summation convention.

\begin{remark}
When working with indexed expressions, we use the Einstein summation
convention: repeated indices, one raised and one lowered, are summed over.
After fixing an orthonormal basis, we freely identify vectors and covectors
using the metric, and hence often write all indices lowered. In that case,
repeated indices are still understood to be summed over.
\end{remark}

\subsection{The groups $G_2$ and $\mathrm{SU}(3)$}

Recall that
\[
G_2=\operatorname{Aut}(\mathbb O).
\]
Every automorphism fixes the identity in \(\mathbb O\), preserves conjugation, and hence preserves the decomposition
\[
\mathbb O=\mathbb R\oplus \operatorname{Im}\mathbb O.
\]
Thus \(G_2\) acts faithfully on \(\operatorname{Im}\mathbb O\), preserving the metric \(g\), the cross product \(\times\), and the positive \(3\)-form \(\varphi\). In this way \(\ImO\) is identified with the standard \(7\)-dimensional irreducible representation of \(G_2\) \cite{SalamonWalpuski2017NotesOctonions,Karigiannis2019IntroductionG2}.

\begin{lemma} \label{lem:G2-stabilizer-SU3}
Let \(\uhat\in \operatorname{Im}\mathbb O\) be a unit vector. The stabilizer of \(\uhat\) in \(G_2\) is isomorphic to \(\mathrm{SU}(3)\).
\end{lemma}

\begin{proof}
Since \(\uhat\) is a unit imaginary octonion, the orthogonal complement \(\uhat^\perp\subset \ImO\) carries a complex structure
\[
J_u(\vhat)=\uhat\times \vhat,
\]
because \(J_u^2=-\mathrm{id}\) on \(\uhat^\perp\). The stabilizer of \(\uhat\) in \(G_2\) preserves \(g\), \(\varphi\), and \(J_u\), hence acts by complex linear isometries on \(\uhat^\perp\cong \mathbb C^3\). Moreover, it preserves the induced complex volume form, so the stabilizer is isomorphic to \(\mathrm{SU}(3)\); see, for example, \cite[Chapter III]{Harvey1990SpinorsCalibrations} or \cite[Section 2]{SalamonWalpuski2017NotesOctonions}.
\end{proof}

Accordingly, after fixing \(\uhat\), the complexified \(G_2\)-representation \(\mathbf{7}_\C\) restricts to \(\mathrm{SU}(3)\) as
\[
\seven_\C\cong \mathbb C\uhat\oplus (\uhat^\perp)_\mathbb C
\cong \mathbf{1}\oplus \mathbf{3}\oplus \overline{\mathbf{3}}
\]
or equivalently, as a real representation,
\[
\seven \cong \mathbf{1}\oplus \mathbf{6}
\]
over \(\mathbb R\).

\subsection{The octonionic extension $W=\mathbb{O}\otimes_{\mathbb{R}}V$}

Let \(V\) be an oriented Euclidean \(7\)-dimensional real vector space, and fix a \(G_2\)-equivariant isometric identification
\[
V \cong \widehat V := \operatorname{Im}\mathbb O.
\]
Via this identification, we transport the metric \(g\), the positive \(3\)-form \(\varphi\), and the cross product from \(\widehat V\) to \(V\). We emphasize, however, that \(V\) is only isomorphic to $\Vhat$ as a $G_2$-module, not as an algebra, and hence octonion multiplication remains only on \(\widehat V\subset \mathbb O\).
We now pass to the octonionic extension
\begin{equation}\label{eq:W-definition}
W:=\mathbb{O}\otimes_{\mathbb{R}}V.
\end{equation}
This is a real vector space of dimension \(56\), naturally equipped with both left and right octonion multiplication on the first factor. For \(a\in \mathbb{O}\) and \(x\otimes v\in W\), define
\begin{equation}\label{eq:left-right-multiplication-W}
L_a(x\otimes v) := (ax)\otimes v,
\qquad
R_a(x\otimes v) := (xa)\otimes v.
\end{equation}
These are well-defined real-linear endomorphisms of \(W\). Since \(\mathbb{O}\) is nonassociative, the operators \(L_a\) and \(R_b\) do not in general satisfy the same formal identities as in the associative setting; nevertheless both actions will play an essential role below.

The Euclidean inner product on \(W\) is the tensor product metric
\begin{equation}\label{eq:W-inner-product}
\langle x\otimes v,\; y\otimes w\rangle_W
=
\langle x,y\rangle_{\mathbb O}\, g(v,w),
\end{equation}
extended bilinearly over the real numbers. With respect to this metric, the \(G_2\)-action on \(\mathbb{O}\) and \(V\) induces an orthogonal action on \(W\), given by
\begin{equation}\label{eq:diagonal-G2-action-W}
\gamma\cdot (x\otimes v) := \gamma(x)\otimes \gamma(v),
\qquad \gamma\in G_2.
\end{equation}
Since \(G_2\) fixes the real line in \(\mathbb{O}\) and acts on \(V\cong \operatorname{Im}\mathbb{O}\) in the standard way, this action preserves all the structures introduced above.

It is often useful to identify \(W\) using the decomposition
\[
\mathbb{O} \cong \mathbb{R}\oplus \Vhat.
\]
Then
\begin{equation}\label{eq:W-as-V-plus-tensor}
W \cong (\mathbb{R}\oplus \Vhat)\otimes V
\cong V \oplus (\Vhat\otimes V).
\end{equation}
This makes transparent the link between the octonionic structure on \(\mathbb{O}\) and the tensor algebra of the standard \(G_2\)-module \(V\).

\subsection{$G_2$-decomposition of $W$}

We next record the decomposition of \(W\) into irreducible \(G_2\)-modules. The required representation-theoretic splittings are standard; see, for example, \cite[Sections 2 and 4]{Karigiannis2019IntroductionG2}, \cite[Section 2]{SalamonWalpuski2017NotesOctonions}, or the classical account in \cite{Bryant1987ExceptionalHolonomy}. Using the fixed \(G_2\)-module identification \(\widehat V\cong V\), we regard
\[
W\cong V\oplus(V\otimes V).
\]
It thus suffices to decompose \(V\otimes V\). For the standard \(7\)-dimensional representation of \(G_2\), one has
\[
V\otimes V
=
S^2V \oplus \Lambda^2V,
\]
with
\[
S^2V \cong \mathbf{1}\oplus \mathbf{27},
\qquad
\Lambda^2V \cong \mathbf{7}\oplus \mathbf{14}.
\]
Accordingly,
\[
W
\cong
\mathbf{7}\oplus
\bigl(\mathbf{1}\oplus \mathbf{27}\oplus \mathbf{7}\oplus \mathbf{14}\bigr),
\]
that is,
\begin{equation}\label{eq:W-as-G2-summands}
W \cong \mathbf{1}\oplus \mathbf{7}\oplus \mathbf{7}\oplus \mathbf{14}\oplus \mathbf{27}.
\end{equation}

We briefly identify these summands.

\begin{itemize}
    \item The \(\mathbf{1}\) summand is the metric line in \(S^2V\), spanned by the tensor \(g\).
    \item One copy of \(\mathbf{7}\) is the initial factor \(V\subset W\cong V\oplus(V\otimes V)\).
    \item The second copy of \(\mathbf{7}\) is the \(\mathbf{7}\)-summand in \(\Lambda^2V\), equivalently the image of the cross product map
    \[
    \Lambda^2V \to V,
    \qquad
    u\wedge v \mapsto u\times v.
    \]
    \item The \(\mathbf{14}\) summand is the kernel of the cross product map \(\Lambda^2V\to V\), and may be identified with the adjoint representation \(\mathfrak{g}_2\).
    \item The \(\mathbf{27}\) summand is the traceless symmetric part
    \[
    S^2_0V = \{T\in S^2V : \operatorname{tr}_g T = 0\}.
    \]
\end{itemize}

This decomposition will be used throughout the paper. It provides the basic \(G_2\)-representation-theoretic framework for the spectral analysis of \(h\), and it will later be refined after fixing a unit vector and restricting to the stabilizer \(\mathrm{SU}(3)\subset G_2\).

\section{Definition and properties of the operator $h$}\label{sec:operator-h}

We now introduce the operator
\[
h:W\to W,
\qquad
W=\mathbb O\otimes_{\mathbb R}V,
\]
which will be the main object of the paper. 

\subsection{Definition of $h$}

Let \(\{e_1,\dots,e_7\}\) be a basis of \(V\) and \(\{\ehat{1},\dots,\ehat{7}\}\) be corresponding basis of $\Vhat$. Also let \(\{\varepsilon^1,\dots,\varepsilon^7\}\) be the dual basis of \(V^*\). For each $i$ and $j$, define 
\begin{equation}\label{eq:h-lowered-components}
h_{ij}=\ehat{i}\ehat{j}=-g_{ij}1+\varphi_{ij}{}^{k}\ehat{k} \in \O.    
\end{equation}
Raising the first index using $g^{ij}$, we can then also define: 
\begin{equation}\label{eq:h-raised-components}
    h^i{}_j=-\delta^i{}_j1 + \varphi^i{}_j{}^k\ehat{k} \in \O
\end{equation}
We then define the \(\mathbb O\)-valued endomorphism tensor
\begin{equation}\label{eq:h-tensor-definition}
h
:=
h^i{}_j\otimes e_i \ \otimes \varepsilon^j
\;\in\;
\mathbb O\otimes V\otimes V^*
\cong
\mathbb O\otimes \operatorname{End}(V).
\end{equation}
Via left multiplication on the \(\mathbb O\)-factor, this tensor induces a real-linear operator
\[
h\in \operatorname{End}_{\mathbb R}(W).
\]

More explicitly, if
\[
\what=\what^j\otimes e_j\in W,
\qquad
\what^j\in \mathbb O,
\]
then \(h(\what)\) is defined by
\begin{equation} \label{eq:hwhat}
    h(\what)=(h^i{}_j\what^j)\tens e_i.
\end{equation}

Thus \(h\) is of left octonionic multiplication type: it is obtained by contracting the \(V\)-index against the octonion multiplication tensor and then acting by left multiplication on the \(\mathbb O\)-factor.

Although this formula is written using a basis, it is tensorial, and hence independent of the choice of basis. 
\begin{remark}
    It is important to note that component-wise, $h$ acts by multiplication by $\ehat{i}\ehat{j}$ . Given orthonormal bases, and identifying raised and lowered indices using the metric, \eqref{eq:hwhat} becomes 
    \begin{equation}\label{eq:hwhat-ortho}
         h(\what)=((\ehat{i}\ehat{j})\what_j)\tens e_i.
    \end{equation}    
Here the order of parenthesization matters due to non-associativity, since $\ehat{i},\ehat{j},\what_j \in \O$.  One can ask what would happen if we instead chose a different parenthesization, for instance $\ehat{i}(\ehat{j}\what_j)$. Indeed, that could be rewritten as $L_{\ehat{i} }L_{\ehat{j}} \what_j$ . However, the operators \(L_{\hat u}\) for \(\hat u\in \ImO \) satisfy the Clifford relation
\[
L_{\hat u}L_{\hat v}+L_{\hat v}L_{\hat u}=-2g(u,v)\id,
\]
so this alternative construction is described by the associative Clifford algebra action on \(\mathbb O\) \cite[Section 2.3]{Baez2002Octonions}. By contrast, the parenthesization in \eqref{eq:hwhat-ortho} retains the genuine octonion product and therefore encodes nonassociative information.
\end{remark}

\subsection{Self-adjointness and equivariance}

We next derive two of the basic structural properties of \(h\): self-adjointness and \(G_2\)-equivariance.

\begin{proposition}\label{prop:h-selfadjoint-equivariant}
The operator \(h\) is self-adjoint with respect to the natural Euclidean inner product on \(W\), and \(G_2\)-equivariant with respect to the diagonal action of \(G_2\) on \(W=\O\tensor_{\R}V\).
\end{proposition}

\begin{proof}
Let
\[
\what=\what^j\tensor e_j,
\qquad
\zhat=\zhat^i\tensor e_i
\]
be elements of \(W\), where \(\what^j,\zhat^i\in \O\). Using the inner product on \(W\) induced from \(\O\) and \(V\), together with the standard identity
\[
\ip{ax}{y}_{\O}=\ip{x}{\bar a\,y}_{\O},
\qquad a,x,y\in \O,
\]
we compute
\begin{align*}
\ip{h(\what)}{\zhat}_W
&=
\ip{h^i{}_j\what^j}{\zhat^i}_{\O} \\
&=
\ip{\what^j}{\overline{h^i{}_j}\,\zhat^i}_{\O}.
\end{align*}
Now
\[
h_{ij}=\ehat{i}\ehat{j},
\]
so since \(\ehat{i},\ehat{j}\in \ImO\), we have
\[
\overline{h_{ij}}
=
\overline{\ehat{i}\ehat{j}}
=
\overline{\ehat{j}}\,\overline{\ehat{i}}
=
\ehat{j}\ehat{i}
=
h_{ji},
\]
and similarly with a raised index. Therefore
\begin{align*}
\ip{h(\what)}{\zhat}_W
&=
 \ip{\what^j}{h^j{}_i\,\zhat^i}_{\O} \\
&=
 \ip{\what^j}{(h(\zhat))^j}_{\O}
=
\ip{\what}{h(\zhat)}_W.
\end{align*}
Hence \(h\) is self-adjoint.

For equivariance, consider first the \(\O\)-valued bilinear form
\[
h=h_{ij}\,\varepsilon^i\otimes \varepsilon^j
\in \O\otimes V^*\otimes V^*,
\qquad
h_{ij}=\ehat{i}\ehat{j}.
\]
The group \(G_2\) acts on \(\O\) by octonion automorphisms and on \(V^*\) by the dual action, so for $a\in\O$ and  $\alpha,\beta \in  V^*$,
\[
\gamma\cdot (a\otimes \alpha\otimes \beta)
=
\gamma(a)\otimes (\alpha\circ \gamma^{-1})\otimes (\beta\circ \gamma^{-1}).
\]
Since the identification \(V\cong \Vhat=\Im\O\) is \(G_2\)-equivariant, if
\[
\gamma(\ehat{i})=\Gamma^p{}_i\ehat{p},
\qquad \text{then} \qquad
\gamma(\varepsilon^i)=(\Gamma^{-1})^i{}_q\,\varepsilon^q.
\]
Hence, using the fact that $\gamma\in \Gtwo$ and hence preserves octonion multiplication, we get
\begin{align*}
\gamma\cdot h
&=
\gamma(h_{ij})\,\gamma(\varepsilon^i)\otimes \gamma(\varepsilon^j) \\
&=
\gamma(\ehat{i}) \gamma(\ehat{j})\,(\Gamma^{-1})^i{}_q(\Gamma^{-1})^j{}_r\,\varepsilon^q\otimes \varepsilon^r \\
&=
\Gamma^p{}_i\Gamma^s{}_j\,h_{ps}\,(\Gamma^{-1})^i{}_q(\Gamma^{-1})^j{}_r\,\varepsilon^q\otimes \varepsilon^r \\
&=
h_{qr}\,\varepsilon^q\otimes \varepsilon^r \\
&=h.
\end{align*}
Thus \(h\in \O\otimes V^*\otimes V^*\) is \(G_2\)-invariant.

Since the mixed tensor \(h^i{}_j=g^{ip}h_{pj}\) is obtained from \(h_{ij}\) by contraction with the \(G_2\)-invariant metric, it is also \(G_2\)-invariant. Therefore the induced operator
\[
h(\what)=(h^i{}_j\what^j)\tensor e_i
\]
is \(G_2\)-equivariant on \(W\).
\end{proof}

As an immediate consequence, the real spectral theory of \(h\) is strongly constrained by the \(G_2\)-decomposition of \(W\) from Section~\ref{sec:background}. In particular, Schur theory implies that \(h\) acts by scalars on the irreducible summands \(\mathbf 1\), \(\mathbf{14}\), and \(\mathbf{27}\), while on the isotypic component \(\mathbf 7\oplus \mathbf 7\) it is represented by a real \(2\times2\) matrix.

\begin{proposition}\label{prop:M-two-dimensional}
Let \(\mathcal M\subset \End_{\R}(W)\) denote the subspace of endomorphisms of left octonionic multiplication type. Then the \(G_2\)-invariant Hermitian subspace \(\mathcal M^{G_2}_{\mathrm{sa}}\) is two-dimensional. Equivalently, after fixing the natural inner product on this space, the normalized \(G_2\)-invariant Hermitian operators form a circle \(S^1\).
\end{proposition}

\begin{proof}
By definition, \(\mathcal M\) is the image of
\[
\Phi:\O\tensor \End(V)\to \End_{\R}(W),
\qquad
\Phi(a\tensor B)=L_a\tensor B.
\]
Hence the \(G_2\)-invariant part identifies with the invariant tensors in
\[
\O\tensor \End(V).
\]
As a \(G_2\)-module one has
\[
\O\cong \mathbf 1\oplus \mathbf 7,
\qquad
\End(V)\cong V\tensor V^*
\cong \mathbf 1\oplus \mathbf 7\oplus \mathbf{14}\oplus \mathbf{27}.
\]
Therefore
\[
(\O\tensor \End(V))^{G_2}
\cong
(\mathbf 1\tensor \mathbf 1)^{G_2}
\oplus
(\mathbf 7\tensor \mathbf 7)^{G_2},
\]
which is two-dimensional, since both \(\mathbf 1\tensor \mathbf 1\) and \(\mathbf 7\tensor \mathbf 7\) contain a unique invariant line.

Concretely, these two invariant directions correspond to the scalar and imaginary parts of octonion multiplication on \(\Vhat=\Im\O\), namely the metric term and the cross-product term. Both induce self-adjoint operators on \(W\), so the Hermitian condition does not reduce the invariant space further. Thus \(\mathcal M^{G_2}_{\mathrm{sa}}\) is a real \(2\)-plane, and its unit sphere is \(S^1\).
\end{proof}

\begin{remark}
The $G_2$-invariant two-plane of Hermitian tensors of left octonionic multiplication type is spanned by the octonion product and the opposite product,
\[
(\hat u,\hat v)\mapsto \hat u\hat v,
\qquad
(\hat u,\hat v)\mapsto \hat v\hat u.
\]
Equivalently, it is spanned by the real and imaginary parts of octonion multiplication:
\[
-g_{ij}\,1
\qquad\text{and}\qquad
\varphi_{ij}{}^{k}\hat e_k.
\]
Thus the general element of this plane may be written as
\begin{equation}\label{eq:invariant-two-plane-hab}
h^{(a,b)}_{ij}
=
-a\,g_{ij}\,1+b\,\varphi_{ij}{}^{k}\hat e_k
=
\frac{a+b}{2}\,\hat e_i\hat e_j+\frac{a-b}{2}\,\hat e_j\hat e_i.
\end{equation}
In this paper, we make the choice $a=1, b=1$. Other choices lead to related pencils, but the corresponding spectral loci may degenerate or change form; we do not analyze that family here.
\end{remark}

\subsection{Analogy with complex structures} \label{sec:cx-analogy}

Although \(h\) is not a complex structure on \(W\), it is useful to keep the complex-geometric analogy in mind. Carrying over the cross product and $\varphi$ to $V$,  for a unit vector \(u\in V \), the map
\[
J_u(v)=u\times v=\varphi_{ij}{}^k u^iv^je_k
\]
defines a complex structure on \(u^\perp\subset V\), and the stabilizer of \(u\) in \(G_2\) is the corresponding \(\mathrm{SU}(3)\)-subgroup. In this sense, the octonion algebra carries a family of slice-wise complex structures parametrized by unit imaginary octonions. 

The operator \(h\) may be viewed as a global \(G_2\)-equivariant package of this family. Indeed, if \(v=v^j e_j\in V\), viewed in \(W\) as \(1\otimes v\), then
\begin{align}
h(1\otimes v)
&=
\bigl(-\delta^i{}_j\,1+\varphi^i{}_j{}^k\ehat{k}\bigr)v^j\otimes e_i \notag \\
&=
-\,1\otimes v+\sum_{i,j,k}\varphi^i{}_j{}^k v^j\,\ehat{k}\otimes e_i \notag \\
&=
-\,1\otimes v-\sum_{k=1}^7 \ehat{k}\otimes J_{e_k}(v), \label{eq:hcxstruct}
\end{align}
where the last equality uses
\[
J_{e_k}(v)^i=\varphi^i{}_{k j}v^j
\]
and the antisymmetry of \(\varphi\) in its lower indices.

 This construction may be compared with the standard complexification of a real vector space equipped with a complex structure. Indeed, suppose $U$ is a $2n$-dimensional real vector space equipped with a complex structure $J$, and suppose $\C \tens U$ is its complexification.   Then the second term in \eqref{eq:hcxstruct} may be compared to taking $i \tens J(x)$ for $x\in U $.   In the present setting, the single complex structure \(J\) is replaced by the \(G_2\)-family \(\{J_{e_k}\}_{k=1}^7\), and \(h\) packages these slice-wise complex structures together with coefficients in the octonion basis \(\{\ehat{k}\}\).

\section{Real spectral theory of $h$}\label{sec:real-spectrum}

We now determine the ordinary real spectrum of the operator
\[
h:W\to W,
\qquad
W=\mathbb O\otimes_{\mathbb R}V.
\]
The basic mechanism is representation-theoretic: since \(h\) is \(G_2\)-equivariant, its action is constrained by the decomposition
\[
W\cong \mathbf 1\oplus \mathbf 7\oplus \mathbf 7\oplus \mathbf{14}\oplus \mathbf{27}.
\]
Thus the computation reduces to scalar actions on the multiplicity-one summands and a single \(2\times2\) matrix on the \(\mathbf 7\oplus \mathbf 7\) isotypic component.

\subsection{Schur-theoretic reduction}

Recall from Section~\ref{sec:background} that
\[
W\cong \mathbf 1\oplus \mathbf 7\oplus \mathbf 7\oplus \mathbf{14}\oplus \mathbf{27}
\]
as a real \(G_2\)-module. Since \(h\) is \(G_2\)-equivariant, by Schur's Lemma, it lies in the $G_2$-action commutant
\[
\End_{G_2}(W)\cong \mathbb R\oplus M_2(\mathbb R)\oplus \mathbb R\oplus \mathbb R.
\]
Accordingly,
\begin{itemize}
    \item \(h\) acts by scalars on \(\mathbf 1\), \(\mathbf{14}\), and \(\mathbf{27}\);
    \item \(h\) acts by a \(2\times2\) real matrix on the two copies of \(\mathbf 7\).
\end{itemize}
Thus the real spectrum is determined by these four blocks.

\subsection{Action on the irreducible components}

Fix an orthonormal basis \(\{e_1,\dots,e_7\}\) of \(V\), and let
\[
\{\ehat{1},\dots,\ehat{7}\}
\]
denote the corresponding orthonormal basis of \(\Vhat=\Im\O\subset \O\). Thus, as before, unhatted basis vectors belong to the abstract \(G_2\)-module \(V\), whereas hatted basis vectors belong to \(\Vhat\) and may be multiplied using the octonion product.

With respect to the chosen basis of \(\Vhat\), octonion multiplication takes the form
\[
\ehat{i}\ehat{j}=-g_{ij}\,1+\varphi_{ijk}\ehat{k}.
\]
Since the basis is orthonormal, we may identify \(g_{ij}=\delta_{ij}\) and raise/lower indices freely. In this subsection, however, we keep all indices downstairs.

Accordingly, the coefficients of \(h\) are
\[
h_{ij}=\ehat{i}\ehat{j}=-g_{ij}\,1+\varphi_{ijk}\ehat{k}.
\]

Now write \(\what\in W=\O\tensor V\) as
\[
\what=\what_j\tensor e_j,
\qquad
\what_j=v_j\,1+w_{ij}\ehat{i},
\]
where \(v_j,w_{ij}\in \R\). Under the identification
\[
W\cong V\oplus (\Vhat\tensor V),
\]
the real and imaginary parts of \(\what\) are thus \((v_j,w_{ij})\).

To explicitly work out the action of $h$ on $\what$, we will use standard contraction identities for the \(G_2\)-forms \(\varphi\) and \(\psi\); see, for example, \cite[Section 2]{Karigiannis2019IntroductionG2} and \cite[Section 2]{SalamonWalpuski2017NotesOctonions}. The following component formula will be used repeatedly.
\begin{lemma}\label{lem:h-master}
With notation as above,
\begin{align}\label{eq:h-master}
h(\what)
&=
\bigl(-v_p+\varphi_{ijp}w_{ij}\bigr)1\tensor e_p \notag \\
& +
\bigl(
\varphi_{iqj}v_j
-w_{iq}
+w_{qi}
-g_{iq}\tr(w)
+\psi_{iq\ell j}w_{\ell j}
\bigr)\ehat{i}\tensor e_q,
\end{align}
where
\[
\tr(w)=w_{pp}.
\]
Equivalently, in components,
\begin{align}
(h(\what))^{\mathrm{re}}_{p}
&=
-v_p+\varphi_{ijp}w_{ij}, \label{eq:h-master-real}\\
(h(\what))^{\mathrm{im}}_{iq}
&=
\varphi_{iqj}v_j
-w_{iq}
+w_{qi}
-g_{iq}\tr(w)
+\psi_{iq\ell j}w_{\ell j}. \label{eq:h-master-imag}
\end{align}
\end{lemma}

\begin{proof}
By definition,
\[
h(\what)=h_{pq}\what_q\tensor e_p
=
(\ehat{p}\ehat{q})(v_q\,1+w_{iq}\ehat{i})\tensor e_p.
\]
Using
\[
\ehat{p}\ehat{q}=-g_{pq}\,1+\varphi_{pqk}\ehat{k},
\]
we expand:
\begin{align*}
h(\what)
&=
\bigl(-g_{pq}\,1+\varphi_{pqk}\ehat{k}\bigr)
\bigl(v_q\,1+w_{iq}\ehat{i}\bigr)\tensor e_p \\
&=
\Bigl(
-g_{pq}v_q\,1
-g_{pq}w_{iq}\ehat{i}
+\varphi_{pqk}v_q\ehat{k}
+\varphi_{pqk}w_{iq}\ehat{k}\ehat{i}
\Bigr)\tensor e_p.
\end{align*}
Now use octonion multiplication again:
\[
\ehat{k}\ehat{i}=-g_{ki}\,1+\varphi_{kim}\ehat{m}.
\]
Substituting this gives
\begin{align*}
h(\what)
&=
\Bigl(
-g_{pq}v_q
-\varphi_{pqk}g_{ki}w_{iq}
\Bigr)1\tensor e_p \\
&\qquad
+
\Bigl(
-g_{pq}w_{iq}\ehat{i}
+\varphi_{pqk}v_q\ehat{k}
+\varphi_{pqk}\varphi_{kim}w_{iq}\ehat{m}
\Bigr)\tensor e_p.
\end{align*}
Since the basis is orthonormal, \(g_{pq}=\delta_{pq}\), so the real part becomes
\[
\bigl(-v_p+\varphi_{ijp}w_{ij}\bigr)1\tensor e_p.
\]
This is \eqref{eq:h-master-real}.

For the imaginary part, relabel dummy indices:
\[
-g_{pq}w_{iq}\ehat{i}\tensor e_p
=
-w_{ip}\ehat{i}\tensor e_p
=
-w_{iq}\ehat{i}\tensor e_q,
\]
and
\[
\varphi_{pqk}v_q\ehat{k}\tensor e_p
=
\varphi_{ijq}v_j\ehat{i}\tensor e_q.
\]
Thus the imaginary part is
\[
\bigl(
\varphi_{jiq}v_j
-w_{iq}
+\varphi_{jqk}\varphi_{k\ell i}w_{\ell j}
\bigr)\ehat{i}\tensor e_q.
\]

Now apply the standard \(G_2\)-contraction identity
\[
\varphi_{abk}\varphi_{kcd}
=
g_{ac}g_{bd}-g_{ad}g_{bc}+\psi_{abcd}.
\]
In the present indices, this gives
\[
\varphi_{qjk}\varphi_{k\ell i}
=
\delta_{q\ell}\delta_{ji}
-\delta_{qi}\delta_{j\ell}
+\psi_{qj\ell i}.
\]
Hence
\begin{align*}
\varphi_{qjk}\varphi_{k\ell i}w_{\ell j}
&=
(\delta_{q\ell}\delta_{ji}
-\delta_{qi}\delta_{j\ell}
+\psi_{qj\ell i})w_{\ell j} \\
&=
w_{qi}
-\delta_{qi}w_{jj}
+\psi_{qj\ell i}w_{\ell j}.
\end{align*}
Using the antisymmetry of \(\psi\) to rewrite the last term, this becomes
\[
w_{qi}
-g_{iq}\tr(w)
+\psi_{iq\ell j}w_{\ell j}.
\]
Substituting this into the imaginary part and rearranging indices, yields
\[
\bigl(
+\varphi_{iqj}v_j
-w_{iq}
+w_{qi}
-g_{iq}\tr(w)
+\psi_{iq\ell j}w_{\ell j}
\bigr)\ehat{i}\tensor e_q,
\]
which is \eqref{eq:h-master-imag}. Combining the real and imaginary parts proves \eqref{eq:h-master}.
\end{proof}

We now specialize \eqref{eq:h-master} to the irreducible \(G_2\)-summands. 

\begin{lemma}\label{lem:h-block-computation}
With respect to the decomposition
\[
W\cong \mathbf 1\oplus \mathbf 7\oplus \mathbf 7\oplus \mathbf{14}\oplus \mathbf{27},
\]
the operator \(h\) acts as follows:
\begin{enumerate}
    \item on \(\mathbf 1\), by the scalar \(-7\);
    \item on \(\mathbf{14}\), by the scalar \(-4\);
    \item on \(\mathbf{27}\), by the scalar \(0\);
    \item on the \(\mathbf 7\oplus \mathbf 7\) isotypic component, by the matrix
\begin{equation}\label{eq:h-seven-block-matrix}
 \begin{pmatrix}
-1 & 6\\
1 & 2
\end{pmatrix}.
\end{equation}
   For the \(\mathbf 7\oplus \mathbf 7\) component, we use the ordered parametrization
\[
(v,u)\in V\oplus V
\longmapsto
v_j\,1\otimes e_j
+
\varphi_{kij}u_k\,\ehat{i}\otimes e_j,
\]
so that the first copy of \(\mathbf 7\) is the real \(V\)-summand and the second copy is the \(\Lambda^2_7V\)-summand with components \(w_{ij}=\varphi_{kij}u_k\).
\end{enumerate}
\end{lemma}

\begin{proof}
We apply Lemma~\ref{lem:h-master} to each irreducible summand in turn.

\medskip
\noindent\textbf{The \(\mathbf 1\)-summand.}
This is generated by the metric tensor. Thus
\[
v_p=0,
\qquad
w_{ij}=g_{ij}=\delta_{ij}.
\]
Substituting into \eqref{eq:h-master-real}, we obtain
\[
(h(\what))^{\mathrm{re}}_{p}
=
-\varphi_{ijp}g_{ij}=0,
\]
since \(\varphi_{ijp}\) is skew in \(i,j\), whereas \(g_{ij}\) is symmetric.

For the imaginary part, \eqref{eq:h-master-imag} gives
\[
(h(\what))^{\mathrm{im}}_{iq}
=
-w_{iq}+w_{qi}-g_{iq}\tr(w)+\psi_{iq\ell j}w_{\ell j}.
\]
Here \(w_{iq}=w_{qi}=g_{iq}\), \(\tr(w)=g_{pp}=7\), and
\[
\psi_{ijq\ell}g_{\ell j}=0
\]
by skew-symmetry of \(\psi\) in \(j,\ell\). Hence
\[
(h(\what))^{\mathrm{im}}_{iq}
=
-g_{iq}+g_{qi}-7\delta_{iq}
=
-7g_{iq}.
\]
Therefore
\[
h(g)=-7g,
\]
so \(h\) acts on \(\mathbf 1\) by the scalar \(-7\).

\medskip 
\noindent\textbf{The \(\mathbf 7\oplus \mathbf 7\) isotypic component.}
There are two natural copies of \(\mathbf 7\).

The first is the \(V\)-summand, corresponding to
\[
w_{ij}=0,
\qquad
v_j \text{ arbitrary.}
\]
Substituting into Lemma~\ref{lem:h-master}, we get
\[
(h(\what))^{\mathrm{re}}_{p}=-v_p,
\qquad
(h(\what))^{\mathrm{im}}_{iq}=\varphi_{iqj}v_j.
\]
Thus
\[
h(v_j\,1\tensor e_j)
=
-v_p\,1\tensor e_p
+
(\varphi_{iqj}v_j)\,\ehat{i}\tensor e_q.
\]
The first term is the negative of the original \(\mathbf 7\)-summand, and the second is equal to the \(\Lambda^2_7\)-copy of \(\mathbf 7\). Hence the first column of the \(2\times2\) block is
\[
\begin{pmatrix}
-1\\
1
\end{pmatrix}.
\]

The second copy of \(\mathbf 7\) is \(\Lambda^2_7V\subset \Lambda^2V\), given by
\[
v_p=0,
\qquad
w_{ij}=\varphi_{kij}u_k
\]
for some \(u_k\in\R\). Since \(w_{ij}\) is skew, \(\tr(w)=0\). Then
\[
(h(\what))^{\mathrm{re}}_{p}
=
\varphi_{ijp}\varphi_{kij}u_k.
\]
Using the standard contraction identity
\[
\varphi_{ijp}\varphi_{kij}=6g_{pk},
\]
we obtain
\[
(h(\what))^{\mathrm{re}}_{p}=6u_p.
\]
Thus the upper-right entry of the matrix is \(6\).

For the imaginary part,
\[
(h(\what))^{\mathrm{im}}_{iq}
=
-w_{iq}+w_{qi}+\psi_{iq\ell j}w_{\ell j},
\]
since \(v_j=0\) and \(\tr(w)=0\). Because \(w_{ij}\) is skew,
\[
-w_{iq}+w_{qi}=-2w_{iq}.
\]
Now \(w_{ij}\in \Lambda^2_7V\), so the standard projector identity gives
\[
\psi_{iq\ell j}w_{\ell j}=4w_{iq}.
\]
Therefore
\[
(h(\what))^{\mathrm{im}}_{iq}=-2w_{iq}+4w_{iq}=2w_{iq}.
\]
Thus the lower-right entry is \(2\), and the full matrix on the \(\mathbf 7\oplus \mathbf 7\) block is
\[
\begin{pmatrix}
-1 & 6\\
1 & 2
\end{pmatrix}.
\]

\medskip
\noindent\textbf{The \(\mathbf{14}\)-summand.}
Here
\[
v_p=0,
\qquad
w_{ij}\in \Lambda^2_{14}V.
\]
Thus \(w_{ij}=-w_{ji}\), \(\tr(w)=0\), and
\[
\varphi_{ijp}w_{ij}=0
\]
by definition of \(\Lambda^2_{14}V\). Hence the real part vanishes:
\[
(h(\what))^{\mathrm{re}}_{p}=0.
\]

For the imaginary part,
\[
(h(\what))^{\mathrm{im}}_{iq}
=
-w_{iq}+w_{qi}+\psi_{iq\ell j}w_{\ell j}.
\]
Since \(w\) is skew,
\[
-w_{iq}+w_{qi}=-2w_{iq}.
\]
Now \(w\in \Lambda^2_{14}V\), so the projection identity is
\[
\psi_{iq\ell j}w_{\ell j}=-2w_{iq}.
\]
Therefore
\[
(h(\what))^{\mathrm{im}}_{iq}=-2w_{iq}-2w_{iq}=-4w_{iq}.
\]
Hence
\[
h|_{\mathbf{14}}=-4\,\id.
\]

\medskip
\noindent\textbf{The \(\mathbf{27}\)-summand.}
Here
\[
v_p=0,
\qquad
w_{ij}=w_{ji},
\qquad
w_{pp}=0.
\]
Since \(w_{ij}\) is symmetric, the real part vanishes:
\[
(h(\what))^{\mathrm{re}}_{p}
=
-\varphi_{ijp}w_{ij}=0.
\]
Also,
\[
-w_{iq}+w_{qi}=0,
\qquad
\tr(w)=0.
\]
Thus the imaginary part reduces to
\[
(h(\what))^{\mathrm{im}}_{iq}
=
\psi_{iq\ell j}w_{\ell j}.
\]
But the contraction of the totally skew tensor \(\psi_{iq\ell j}\) against the symmetric tensor \(w_{\ell j}\) vanishes identically, so
\[
(h(\what))^{\mathrm{im}}_{iq}=0.
\]
Hence
\[
h|_{\mathbf{27}}=0.
\]
This completes the block computation.
\end{proof}

\subsection{The real spectrum}

We now assemble the block calculations into the main result of this section.

\begin{theorem}\label{thm:real-spectrum}
The real spectrum of \(h\) is
\begin{equation}\label{eq:real-spectrum-set}
\Spec_{\R}(h)
=
\left\{
-7,\,-4,\,0,\,
\frac{1-\sqrt{33}}{2},\,
\frac{1+\sqrt{33}}{2}
\right\}.
\end{equation}
More precisely:
\begin{enumerate}
    \item \(-7\) occurs on \(\mathbf 1\), with multiplicity \(1\);
    \item \(-4\) occurs on \(\mathbf{14}\), with multiplicity \(14\);
    \item \(0\) occurs on \(\mathbf{27}\), with multiplicity \(27\);
    \item the remaining two eigenvalues are the eigenvalues of
    \[
\begin{pmatrix}
-1 & 6\\
1 & 2
\end{pmatrix}.
    \]
    on the \(\mathbf 7\oplus \mathbf 7\) isotypic component, each with multiplicity \(7\).
\end{enumerate}
\end{theorem}

\begin{proof}
By Lemma~\ref{lem:h-block-computation}, it remains only to diagonalize the \(2\times 2\) block
\[
\begin{pmatrix}
-1 & 6\\
1 & 2
\end{pmatrix}.
\]
Its characteristic polynomial is
\[
\det
\begin{pmatrix}
-1-\lambda & 6\\
1 & 2-\lambda
\end{pmatrix}
=
(-1-\lambda)(2-\lambda)-6
=
\lambda^2-\lambda-8.
\]
Hence the two eigenvalues are
\[
\lambda=\frac{1\pm\sqrt{33}}{2}.
\]
Together with the scalar eigenvalues on \(\mathbf 1\), \(\mathbf{14}\), and \(\mathbf{27}\), this gives the full real spectrum.
\end{proof}
\section{Octonionic right-eigenvalue problem}\label{sec:right-spectrum-reduction}

We now turn to the octonionic right-spectrum of \(h\). The key point of this section is that, after fixing a unit imaginary octonion, the right-eigenvalue equation can be rewritten as an ordinary real eigenvalue problem for a deformed operator. This reduction is the starting point for the later \(\SU(3)\)-equivariant block analysis.

\subsection{Right eigenvalues and the operator \(H_{u,Q}\)}

Recall from \eqref{eq:left-right-multiplication-W} that \(W=\O\tensor_{\R}V\) admits right octonion multiplication via right multiplication on the \(\O\)-factor:
\[
R_a(x\tensor v)=(x a)\tensor v,
\qquad
a,x\in \O,\ v\in V.
\]
Since octonion multiplication is noncommutative, and \(h\) is defined using left multiplication on the \(\O\)-factor, the right-eigenvalue equation
\[
h(\what)=\what \lambda,
\qquad
\what \in W,\ \lambda\in \O,
\]
is nontrivial.

\begin{remark}
There is a simple formal reason why Hermitian octonionic right-eigenvalue problems can differ from the associative case. Suppose, heuristically, that \(H=(h^i{}_j)\) is an octonionic Hermitian matrix and that
\[
h^i{}_j \uhat^j=\uhat^i\lambda
\]
for a nonzero \(\uhat\in W\). Contracting with \(\bar{\uhat}_i\) gives
\[
\bar{\uhat}_i(h^i{}_j \uhat^j)=|\uhat|^2\lambda.
\]
Applying conjugation and the Hermitian symmetry $\overline{h^i{}_j}=h_j{}^i$ allows one to obtain 
\[
|\uhat|^2(\lambda-\bar\lambda)
=
[\bar{\uhat}_i,h^i{}_j,\uhat^j],
\]
where $[a,b,c]=a(bc)-(ab)c$ is the associator. Since for octonions this does not necessarily vanish, we see that $\lambda$ does not need to be real. The same computation may be carried out in the associative case, but the associator vanishes of course, and we obtain that $\lambda$ is real. 
\end{remark}

We are interested in right eigenvalues lying in a fixed complex slice of \(\O\). Let \(u\in V\) be a unit vector, and let \(\uhat\in \Vhat=\Im\O\) be its image under the fixed \(G_2\)-equivariant identification \(V\cong \Vhat\). The corresponding complex slice is
\[
\C_{\uhat}:=\R\oplus \R \uhat \subset \O.
\]
Any \(\lambda\in \C_{\uhat}\) may be written uniquely as
\[
\lambda=\mu+Q\uhat,
\qquad
\mu,Q\in \R.
\]

Accordingly, we introduce the family of real-linear operators
\begin{equation}\label{eq:Hua-def}
H_{u,Q}:=h-Q\,R_{\uhat}\in \End_{\R}(W).
\end{equation}

The point is that the right-eigenvalue equation for \(\lambda=\mu+Q\uhat\) is equivalent to an ordinary real eigenvalue equation for \(H_{u,Q}\).

\begin{proposition}\label{prop:right-eigenvalue-reduction}
Let \(u\in V\) be a unit vector, let \(\uhat\in \Vhat\) denote the corresponding unit imaginary octonion, and let
\[
\lambda=\mu+Q\uhat\in \C_{\uhat},
\qquad
\mu,Q\in \R.
\]
Then, for \(\what\in W\), the right-eigenvalue equation
\begin{equation}\label{eq:right-eigenvalue}
h(\what)=\what\lambda
\end{equation}
is equivalent to
\begin{equation}\label{eq:Hua-eigenvalue}
H_{u,Q}(\what)=\mu \what.
\end{equation}
Equivalently,
\[
\ker(h-R_\lambda)\neq 0
\quad\Longleftrightarrow\quad
\ker(H_{u,Q}-\mu\,\id)\neq 0.
\]
\end{proposition}

\begin{proof}
Write
\[
\lambda=\mu+Q\uhat.
\]
Then
\[
\what\lambda=\what(\mu+Q\uhat)=\mu \what+Q\,\what\uhat=\mu \what+Q\,R_{\uhat}(\what),
\]
since \(\mu\in \R\) is central and \(R_{\uhat}\) denotes right multiplication by \(\uhat\).

Thus \eqref{eq:right-eigenvalue} is equivalent to
\[
h(\what)=\mu \what+Q\,R_{\uhat}(\what),
\]
which in turn is equivalent to
\[
(h-Q\,R_{\uhat})(\what)=\mu \what.
\]
By definition of \(H_{u,Q}\), this is exactly \eqref{eq:Hua-eigenvalue}.
\end{proof}

\begin{remark}
The advantage of Proposition~\ref{prop:right-eigenvalue-reduction} is that it separates the genuinely octonionic part of the problem, encoded by the right multiplication operator \(R_{\uhat}\), from the real spectral parameter \(\mu\). Thus, after fixing the slice \(\C_{\uhat}\), the search for octonionic right eigenvalues reduces to the study of the real spectrum of the family \(H_{u,Q}\).
\end{remark}

\subsection{$\SU(3)$-equivariance}

We now explain why the operators \(H_{u,Q}\) inherit a residual symmetry after fixing a unit vector \(u\in V\). As in Section~\ref{sec:background}, recall that
\[
\SU(3)\subset G_2
\]
is the stabilizer of \(u\), equivalently of the corresponding unit imaginary octonion \(\uhat\in \Vhat=\Im\O\).

The key observation is that both ingredients in the definition
\[
H_{u,Q}=h-Q\,R_{\uhat}
\]
are \(\SU(3)\)-equivariant: the operator \(h\) is already \(G_2\)-equivariant by Proposition~\ref{prop:h-selfadjoint-equivariant}, while \(R_{\uhat}\) is \(\SU(3)\)-equivariant because elements of \(\SU(3)\) fix \(\uhat\).

\begin{proposition}\label{prop:Hua-su3-equivariant}
Let \(u\in V\) be a unit vector, let \(\uhat\in \Vhat\) denote the corresponding unit imaginary octonion, and let \(Q\in \R\). Then the operator
\[
H_{u,Q}=h-Q\,R_{\uhat}
\]
is \(\SU(3)\)-equivariant with respect to the diagonal action of the stabilizer \(\SU(3)\subset G_2\) of \(u\) on \(W=\O\tensor_{\R}V\).
\end{proposition}

\begin{proof}
Let \(\gamma\in \SU(3)\subset G_2\). Since \(\SU(3)\) is the stabilizer of \(u\), and the identification \(V\cong \Vhat\) is \(G_2\)-equivariant, we have
\[
\gamma(\uhat)=\uhat.
\]
By Proposition~\ref{prop:h-selfadjoint-equivariant}, the operator \(h\) is \(G_2\)-equivariant, hence in particular
\[
h(\gamma\cdot \what)=\gamma\cdot h(\what)
\]
for all \(\what\in W\).

It therefore remains to check that \(R_{\uhat}\) is \(\SU(3)\)-equivariant. Let \(\what=\what^q\tensor e_q\in W\). Then
\[
R_{\uhat}(\what)=(\what^q\uhat)\tensor e_q.
\]
Applying \(\gamma\) gives
\[
\gamma\cdot R_{\uhat}(\what)
=
\gamma(\what^q\uhat)\tensor \gamma(e_q).
\]
Since \(\gamma\in G_2=\Aut(\O)\) is an octonion automorphism,
\[
\gamma(\what^q\uhat)=\gamma(\what^q)\gamma(\uhat)=\gamma(\what^q)\uhat,
\]
because \(\gamma(\uhat)=\uhat\). Hence
\[
\gamma\cdot R_{\uhat}(\what)
=
(\gamma(\what^q)\uhat)\tensor \gamma(e_q)
=
R_{\uhat}(\gamma\cdot \what).
\]
Thus \(R_{\uhat}\) commutes with the \(\SU(3)\)-action. Since both \(h\) and \(R_{\uhat}\) are \(\SU(3)\)-equivariant, so is
\[
H_{u,Q}=h-Q\,R_{\uhat}.
\]
\end{proof}

\begin{remark}
The symmetry group drops from \(G_2\) to \(\SU(3)\) precisely because the operator \(R_{\uhat}\) depends on the choice of the distinguished imaginary unit \(\uhat\). Thus the family \(H_{u,Q}\) is no longer \(G_2\)-equivariant in general, but retains the full symmetry of the stabilizer of the chosen complex slice
\[
\C_{\uhat}=\R\oplus \R\uhat.
\]
This is the representation-theoretic reason that the right-eigenvalue problem will be governed by the \(\SU(3)\)-decomposition of \(W\).
\end{remark}

\subsection{$\SU(3)$-decomposition of $W$}

\subsubsection{The basic $\SU(3)$-representations associated to $u$}

Fix a unit vector \(u\in V\), and let \(\uhat\in \Vhat=\Im\O\) denote the corresponding unit imaginary octonion. As in Section~\ref{sec:background}, the cross product with \(u\) defines an orthogonal complex structure
\begin{equation}\label{eq:slice-complex-structure}
J_u:u^\perp\to u^\perp,
\qquad
J_u(v)=u\times v.
\end{equation}
Indeed, if \(v\in u^\perp\), then
\[
J_u^2(v)=u\times(u\times v)=-v,
\]
so \(u^\perp\) becomes a complex \(3\)-dimensional Hermitian vector space. The stabilizer of \(u\) in \(G_2\) preserves \(g\), \(J_u\), and the induced complex volume form on \(u^\perp\), and is therefore isomorphic to \(\SU(3)\).

We now describe the real \(\SU(3)\)-representations that will appear in the restriction of the \(G_2\)-modules \(\seven\), \(\fourteen\), and \(\twentyseven\). The standard notation and basic structure of the \(\SU(3)\)-representations used below may be found, for example, in \cite[Chapter~14]{FultonHarris1991RepresentationTheory}; for the \(G_2\)-geometric realization via the complex structure \(J_u\), see also \cite[Section~2]{SalamonWalpuski2017NotesOctonions}.

\begin{definition}
With \(u\) fixed as above, let:
\begin{enumerate}
    \item \(\one\) denote the trivial real representation \(\R u\);
    \item \(\six\) denote the real representation \(u^\perp\), viewed as a complex \(3\)-dimensional Hermitian space via \(J_u\), i.e. $\three\oplus\threebar$;
    \item \(\eight\) denote the adjoint representation \(\mathfrak{su}(u^\perp,J_u)\);
    \item \(\mathbf{12}\) denotes the real representation whose complexification is \(S^2(\three)\oplus S^2(\threebar)\).
\end{enumerate}
\end{definition}

It is useful to keep in mind the following concrete geometric models.

\begin{remark}
The representation \(\six=u^\perp\) is the basic \(\SU(3)\)-module arising from the orthogonal splitting
\[
V=\R u\oplus u^\perp.
\]
The representation \(\eight\) is the space of skew-adjoint endomorphisms of \(u^\perp\) commuting with \(J_u\) and having complex trace zero:
\[
\eight
\cong
\{T\in \mathfrak{so}(u^\perp): TJ_u=J_uT,\ \tr_{\C}(T)=0\}.
\]
Finally, \(\mathbf{12}\) may be identified with the space of trace-free symmetric bilinear forms on \(u^\perp\) that are \(J_u\)-anti-invariant, equivalently those \(s\in S^2_0(u^\perp)\) satisfying
\[
s(J_u v,J_u w)=-s(v,w).
\]
\end{remark}

The significance of these descriptions is that they isolate the \(\SU(3)\)-types that will govern the right-eigenvalue problem. In the next subsection, we explain how these representations arise by restricting the \(G_2\)-modules \(\seven\), \(\fourteen\), and \(\twentyseven\) to the stabilizer of \(u\).

\subsubsection{Restriction of the \(G_2\)-modules \(\seven\), \(\fourteen\), and \(\twentyseven\)}

We now describe how the basic \(G_2\)-modules appearing in the decomposition of \(W\) restrict to the stabilizer \(\SU(3)\subset G_2\) of the fixed unit vector \(u\in V\). The discussion is guided by the geometric models introduced above.

\begin{proposition}\label{prop:G2-to-SU3-branching}
Under restriction from \(G_2\) to \(\SU(3)\), the irreducible \(G_2\)-modules \(\seven\), \(\fourteen\), and \(\twentyseven\) decompose as
\begin{equation}\label{eq:G2-to-SU3-branching}
\seven \cong \one\oplus \six,
\qquad
\fourteen \cong \eight\oplus \six,
\qquad
\twentyseven \cong \one\oplus \six\oplus \eight\oplus \mathbf{12}.
\end{equation}
\end{proposition}

\begin{proof}
We treat the three \(G_2\)-modules in turn.

\medskip
\noindent\textbf{The standard representation \(\seven\).}
By construction,
\[
V=\R u\oplus u^\perp.
\]
The line \(\R u\) is fixed by \(\SU(3)\), hence is the trivial representation \(\one\). The orthogonal complement \(u^\perp\), equipped with the complex structure \(J_u\), is the real \(6\)-dimensional representation \(\six\). Thus
\[
\seven \cong \one\oplus \six.
\]

\medskip
\noindent\textbf{The adjoint representation \(\fourteen\).}
Recall that
\[
\fourteen\cong \mathfrak g_2.
\]
Under restriction to the stabilizer of \(u\), the Lie algebra \(\mathfrak g_2\) decomposes as
\[
\mathfrak g_2=\mathfrak{su}(3)\oplus \mathfrak m,
\]
where \(\mathfrak{su}(3)\) is the Lie algebra of the stabilizer and \(\mathfrak m\) is its orthogonal complement with respect to the Killing form. The first summand is the adjoint representation \(\eight\). The second is \(6\)-dimensional and may be identified, as an \(\SU(3)\)-module, with \(u^\perp\), hence with \(\six\). Therefore
\[
\fourteen \cong \eight\oplus \six.
\]

\medskip
\noindent\textbf{The traceless symmetric representation \(\twentyseven\).}
Recall that
\[
\twentyseven\cong S^2_0V.
\]
Using the orthogonal decomposition
\[
V=\R u\oplus u^\perp,
\]
we obtain
\[
S^2V
\cong
S^2(\R u)\oplus (u\odot u^\perp)\oplus S^2(u^\perp),
\]
where \(\odot\) denotes symmetrized tensor product. Passing to the traceless part gives
\[
S^2_0V
\cong
\one\oplus \six\oplus S^2_0(u^\perp),
\]
since \(S^2(\R u)\) only contains a \(\one\)-component, and \(u\odot u^\perp \cong \six\). It therefore remains to decompose \(S^2_0(u^\perp)\) as an \(\SU(3)\)-module.

Since \(u^\perp\) is a Hermitian complex \(3\)-space with complex structure \(J_u\), the real traceless symmetric tensors on \(u^\perp\) split into their \(J_u\)-invariant and \(J_u\)-anti-invariant parts:
\[
S^2_0(u^\perp)
\cong
\bigl(S^2_0(u^\perp)\bigr)^{J_u}
\oplus
\bigl(S^2_0(u^\perp)\bigr)^{-J_u}.
\]
The \(J_u\)-invariant traceless symmetric tensors are precisely the traceless Hermitian forms on \(u^\perp\), and these identify with the adjoint representation \(\eight\). The \(J_u\)-anti-invariant traceless symmetric tensors form the real \(12\)-dimensional representation \(\mathbf{12}\). Hence
\[
S^2_0(u^\perp)\cong \eight\oplus \mathbf{12},
\]
and therefore
\[
\twentyseven\cong \one\oplus \six\oplus \eight\oplus \mathbf{12}.
\]
This completes the proof.
\end{proof}

\begin{remark}
The four summands in the restriction of \(\twentyseven\cong S^2_0V\) admit concrete descriptions:
\begin{align*}
\one
&\cong
\Span\!\left\{\,u\odot u-\frac{1}{6}g|_{u^\perp}\right\},\\[4pt]
\six
&\cong
u\odot u^\perp,\\[4pt]
\eight
&\cong
\{\,s\in S^2_0(u^\perp): s(J_uv,J_uw)=s(v,w)\,\},\\[4pt]
\mathbf{12}
&\cong
\{\,s\in S^2_0(u^\perp): s(J_uv,J_uw)=-s(v,w)\,\}.
\end{align*}
These concrete models will be useful in the later block computations. Note that in the $\eight$ and  $\twelve$ representations, $s \in  S^2_0(u^\perp)$. To embed it in $S^2_0V$, we just extend it by $0$ in the $u$-direction. Equivalently, we require $u\lrcorner s=0$. 
\end{remark}
\begin{lemma}\label{lem:8-to-8}
Let
\[
s\in \eight \subset \twentyseven \cong S^2_0V,
\]
so that \(s\) is a traceless symmetric bilinear form on \(u^\perp\) satisfying
\[
s(J_uv,J_uw)=s(v,w)
\qquad
\text{for all }v,w\in u^\perp.
\]
Define a bilinear form \(\omega_s\) on  \(u^\perp\) by
\begin{equation}\label{eq:omega-s-definition}
\omega_s(v_1,v_2):=s(J_uv_1,v_2).
\end{equation}
Then \(\omega_s\) is skew-symmetric, \(J_u\)-invariant, and therefore determines an element
\[
\omega_s\in \eight\subset \fourteen\cong \Lambda^2_{14}V.
\]
\end{lemma}

\begin{proof}
Since \(s\) is \(J_u\)-invariant, we have
\[
s(J_uv_1,J_uv_2)=s(v_1,v_2)
\]
for all \(v_1,v_2\in u^\perp\). Using \(J_u^2=-\id\), it follows that
\[
s(J_uv_1,v_2)
=
-s(v_1,J_uv_2).
\]
Hence
\[
\omega_s(v_1,v_2)=s(J_uv_1,v_2)=-s(J_u v_2,v_1)=-\omega_s(v_2,v_1),
\]
since \(s\) is symmetric. Thus \(\omega_s\) is skew-symmetric.

Next,
\[
\omega_s(J_uv_1,J_uv_2)
=
s(J_u^2v_1,J_uv_2)
=
-s(v_1,J_uv_2)
=
s(J_uv_1,v_2)
=
\omega_s(v_1,v_2),
\]
so \(\omega_s\) is \(J_u\)-invariant.

Therefore \(\omega_s\) is a skew \(2\)-form on \(u^\perp\) of type \((1,1)\). Since \(s\) is traceless, the corresponding \(2\)-form is primitive, and hence lies in the adjoint representation
\[
\eight\subset \mathfrak{su}(u^\perp,J_u)\subset \fourteen.
\]
\end{proof}
\subsubsection{The \(\SU(3)\)-decomposition of \(W\)}

We now combine the preceding branching rules to obtain the \(\SU(3)\)-decomposition of the full space
\[
W=\O\tensor_{\R}V.
\]
Recall that, as a \(G_2\)-module,
\[
W\cong \one\oplus \seven\oplus \seven\oplus \fourteen\oplus \twentyseven.
\]
Restricting each summand to \(\SU(3)\) and using Proposition~\ref{prop:G2-to-SU3-branching}, we immediately obtain the isotypic decomposition of \(W\).

\begin{proposition}\label{prop:W-su3-decomposition}
As a real \(\SU(3)\)-module,
\begin{equation}\label{eq:W-su3-decomp}
W
\cong
\one^{\oplus 4}\oplus \six^{\oplus 4}\oplus \eight^{\oplus 2}\oplus \twelve.
\end{equation}
\end{proposition}

\begin{proof}
Using
\[
W\cong \one\oplus \seven\oplus \seven\oplus \fourteen\oplus \twentyseven,
\]
together with the restriction formulas
\[
\seven \cong \one\oplus \six,
\qquad
\fourteen \cong \eight\oplus \six,
\qquad
\twentyseven \cong \one\oplus \six\oplus \eight\oplus \twelve,
\]
we obtain
\begin{align*}
W
&\cong
\one
\oplus
(\one\oplus \six)
\oplus
(\one\oplus \six)
\oplus
(\eight\oplus \six)
\oplus
(\one\oplus \six\oplus \eight\oplus \twelve) \\
&\cong
\one^{\oplus 4}\oplus \six^{\oplus 4}\oplus \eight^{\oplus 2}\oplus \twelve.
\end{align*}
\end{proof}

The decomposition \eqref{eq:W-su3-decomp} is the representation-theoretic reason that the right-eigenvalue problem breaks into four blocks. Since the operators
\[
H_{u,Q}=h-Q\,R_{\uhat}
\]
are \(\SU(3)\)-equivariant by Proposition~\ref{prop:Hua-su3-equivariant}, each \(H_{u,Q}\) preserves the isotypic summands in \eqref{eq:W-su3-decomp}. Consequently, \(H_{u,Q}\) is determined by its action on the corresponding multiplicity spaces.

\begin{remark}\label{rem:W-su3-blocks}
The decomposition \eqref{eq:W-su3-decomp} yields four blocks:
\begin{enumerate}
    \item a \(4\times 4\) block on the multiplicity space of \(\one^{\oplus 4}\);
    \item a \(4\times 4\) block on the multiplicity space of \(\six^{\oplus 4}\);
    \item a \(2\times 2\) block on the multiplicity space of \(\eight^{\oplus 2}\);
    \item a scalar block on the irreducible summand \(\twelve\).
\end{enumerate}
Thus the octonionic right-eigenvalue problem on the fixed slice
\[
\C_{\uhat}=\R\oplus \R\uhat
\]
reduces to four explicit finite-dimensional calculations. This is the basic mechanism behind the block analysis in the next section.
\end{remark}

It is useful to keep track of the origins of these multiplicity spaces inside the \(G_2\)-decomposition of \(W\):
\begin{itemize}
    \item the four copies of \(\one\):
    \(
    \one_{\Gtwo} \subset W, \text{ two from }  \seven,\text{one from }\twentyseven;
    \)
    \item the four copies of \(\six\):
    \(
    \text{ two from }\seven, \text{ one from } \fourteen,\text{ one from } \twentyseven;
    \)
    \item the two copies of \(\eight\) come from
    \(
     \fourteen \text{ and } \twentyseven;
    \)
    \item the \(\twelve\) occurs only inside \(\twentyseven\cong S^2_0V\).
\end{itemize}
This bookkeeping will be useful when constructing explicit bases for the later block computations.

\section{Block analysis of \(H_{u,Q}\)}\label{sec:block-analysis}

We now analyze the operators \(H_{u,Q}=h-Q R_{\uhat}\) block-by-block with respect to the \(\SU(3)\)-decomposition
\[
W\cong \one^{\oplus 4}\oplus \six^{\oplus 4}\oplus \eight^{\oplus 2}\oplus \twelve.
\]
By Proposition~\ref{prop:right-eigenvalue-reduction}, for the operator $h$, a right eigenvalue 
\(\lambda=\mu+Q\uhat\) with \(Q\neq0\) corresponds to an ordinary real
eigenvalue \(\mu\) of the operator \(H_{u,Q}\).  Thus, on each \(\SU(3)\)-isotypic block, we determine whether \(H_{u,Q}\) has real eigenvalues, and when it does, compute the corresponding values of \(\mu\).

The action of \(h\) itself has already been determined from the \(G_2\)-decomposition in Section~\ref{sec:real-spectrum}. The new ingredient is the symmetry-breaking term \(R_{\uhat}\). Indeed, \(h\) preserves \(G_2\)-types, while \(R_{\uhat}\) preserves only the \(\SU(3)\)-isotypic decomposition associated to the chosen unit \(\uhat\). Since a fixed \(\SU(3)\)-type may occur in several different \(G_2\)-summands, \(R_{\uhat}\) can mix these \(G_2\)-components. This mixing is precisely what the block computations below record.

\subsection{The \(\twelve\)-block}

We begin with the \(\twelve\)-summand. This is the simplest block, because it occurs with multiplicity one and lies entirely inside
\[
\twentyseven\cong S^2_0V.
\]
Recall that, after fixing the unit vector \(u\in V\), the orthogonal complement \(u^\perp\) carries the complex structure
\[
J=J_u,\qquad J(v)=u\times v.
\]
The \(\twelve\)-summand is the space of traceless symmetric tensors on \(u^\perp\) that are \(J\)-anti-invariant:
\[
\twelve
=
\left\{
s\in S^2_0(u^\perp):
s(Jv,Jw)=-s(v,w)
\right\}.
\]
We regard such an \(s\) as an element of \(S^2_0V\) by extending it by zero in the \(u\)-direction; equivalently,
\[
u\lrcorner s=0.
\]

\begin{lemma}\label{lem:twelve-block}
On the \(\twelve\)-summand, the operator \(H_{u,Q}\) acts as
\begin{equation}\label{eq:twelve-block-H}
H_{u,Q}|_{\twelve}
=
-Q R_{\uhat}.
\end{equation}
Moreover, with respect to the complex structure induced by right multiplication by \(\uhat\), this block has real eigenvalues of $H_{u,Q}$ only when \(Q=0\). Consequently, for \(Q\neq0\), the \(\twelve\)-summand contributes no
non-real right eigenvalues \(\lambda=\mu+Q\uhat\) of \(h\).
\end{lemma}

\begin{proof}
Let
\[
s\in \twelve\subset S^2_0V\subset W.
\]
In components, this means that \(s\) is represented by an element
\[
\what=\what_q\tensor e_q,
\qquad
\what_q=s_{iq}\ehat{i},
\]
with
\[
s_{iq}=s_{qi},
\qquad
s_{pp}=0,
\qquad
u^q s_{iq}=0,
\qquad
s(Jv,Jw)=-s(v,w).
\]
Since \(s\in\twelve\subset\twentyseven\), Lemma~\ref{lem:h-block-computation}
gives \(h(s)=0\). Therefore
\[
H_{u,Q}(s)=-Q R_{\uhat}(s).
\]

It remains to note that \(R_{\uhat}\) preserves \(\twelve\). If
\(\hat x\in\Vhat\cap\uhat^\perp\), then
\[
\hat x\uhat=-\widehat{Jx}.
\]
Thus \(R_{\uhat}\) applies \(-J\) to the octonion coefficient index of \(s\).
Since \(s\) is symmetric and \(J\)-anti-invariant, it satisfies
\[
s(Jx,y)=s(x,Jy),
\]
so applying \(J\) to one index again gives a symmetric \(J\)-anti-invariant
tensor. Hence \(R_{\uhat}\) preserves \(\twelve\).

By alternativity of octonion multiplication,
\[
R_{\uhat}^2=R_{\uhat^2}=-\id,
\]
so \(R_{\uhat}\) defines a complex structure on \(\twelve\). Therefore
\[
(H_{u,Q}|_{\twelve})^2=-Q^2\id.
\]
For \(Q\neq0\), its eigenvalues are purely imaginary after complexification, so it has no real eigenvalues. By Proposition~\ref{prop:right-eigenvalue-reduction}, the \(\twelve\)-block contributes no non-real right eigenvalues of \(h\).
\end{proof}
\subsection{The \(\one^{\oplus 4}\)-block}

We next analyze the \(\one^{\oplus 4}\)-isotypic component. In this block it is useful to choose a basis adapted to the \(G_2\)-decomposition of \(W\). Let
\begin{equation}\label{eq:one-block-basis}
\begin{aligned}
\eta_1&:=1\tensor u,\\
\eta_2&:=g_{iq}\,\ehat{i}\tensor e_q,\\
\eta_3&:=\varphi_{iqk}u_k\,\ehat{i}\tensor e_q,\\
\eta_4&:=\left(u_i u_q-\frac17 g_{iq}\right)\ehat{i}\tensor e_q.
\end{aligned}
\end{equation}
Then \(\eta_1\) lies in one copy of the \(\seven\)-summand, \(\eta_2\) lies in the \(\one\)-summand, \(\eta_3\) lies in the other copy of the \(\seven\)-summand, and \(\eta_4\) lies in the \(\one\)-summand inside \(\twentyseven\).

Thus a general element of the \(\one^{\oplus 4}\)-block is written as
\[
\what=A\eta_1+B\eta_2+C\eta_3+D\eta_4.
\]

\begin{remark}\label{rem:one-block-basis-not-orthonormal}
The basis \((\eta_1,\eta_2,\eta_3,\eta_4)\) is adapted to the
\(G_2\)-decomposition, but it is not orthonormal.  The four vectors are
mutually orthogonal, and their squared norms are
\[
\|\eta_1\|^2=1,\qquad
\|\eta_2\|^2=7,\qquad
\|\eta_3\|^2=6,\qquad
\|\eta_4\|^2=\frac67.
\]
Consequently, if
\[
\what=A\eta_1+B\eta_2+C\eta_3+D\eta_4,
\]
then
\[
\|\what\|^2
=
A^2+7B^2+6C^2+\frac67D^2.
\]
In particular, matrices written in this basis need not appear symmetric,
even when they represent self-adjoint operators with respect to the natural
inner product.
\end{remark}

\begin{lemma}\label{lem:one-block}
With respect to the ordered basis
\[
(\eta_1,\eta_2,\eta_3,\eta_4),
\]
the operator \(H_{u,Q}=h-Q R_{\uhat}\) is represented on the \(\one^{\oplus4}\)-block by the matrix
\begin{equation}\label{eq:one-block-matrix}
M_{\one}(Q)
=
\begin{pmatrix}
-1 & Q & 6 & \frac67 Q\\[2pt]
-\frac17 Q & -7 & \frac67 Q & 0\\[2pt]
1 & -Q & 2 & \frac17 Q\\[2pt]
-Q & 0 & -Q & 0
\end{pmatrix}.
\end{equation}
Equivalently, if
\[
\what=A\eta_1+B\eta_2+C\eta_3+D\eta_4,
\]
then
\[
H_{u,Q}(\what)
=
A'\eta_1+B'\eta_2+C'\eta_3+D'\eta_4,
\]
where
\[
\begin{aligned}
A'&=-A+QB+6C+\frac67QD,\\
B'&=-\frac17QA-7B+\frac67QC,\\
C'&=A-QB+2C+\frac17QD,\\
D'&=-QA-QC.
\end{aligned}
\]
\end{lemma}

\begin{proof}
We compute the two pieces \(h\) and \(R_{\uhat}\) separately.

First, the action of \(h\) is already known from the real spectral calculation in Section~\ref{sec:real-spectrum}. In the present \(G_2\)-adapted basis, one has
\[
h(\eta_1)=-\eta_1+\eta_3,
\qquad
h(\eta_3)=6\eta_1+2\eta_3,
\]
because \(\eta_1\) and \(\eta_3\) span the two copies of the \(\seven\)-summand. Also,
\[
h(\eta_2)=-7\eta_2,
\qquad
h(\eta_4)=0,
\]
since \(\eta_2\) lies in the \(\one\)-summand and \(\eta_4\) lies in the \(\twentyseven\)-summand. Therefore
\[
[h]_{\one^{\oplus4}}
=
\begin{pmatrix}
-1 & 0 & 6 & 0\\
0 & -7 & 0 & 0\\
1 & 0 & 2 & 0\\
0 & 0 & 0 & 0
\end{pmatrix}.
\]

It remains to compute the action of \(R_{\uhat}\). We claim that

\begin{subequations}\label{eq:Ru-one}
\begin{align}
R_{\uhat}\eta_1&=\frac17\eta_2+\eta_4, \label{eq:Ru-one-1}\\
R_{\uhat}\eta_2&=-\eta_1+\eta_3, \label{eq:Ru-one-2}\\
R_{\uhat}\eta_3&=-\frac67\eta_2+\eta_4, \label{eq:Ru-one-3}\\
R_{\uhat}\eta_4&=-\frac67\eta_1-\frac17\eta_3. \label{eq:Ru-one-4}
\end{align}
\end{subequations}

Indeed,
\[
R_{\uhat}\eta_1
=
u_q\uhat\tensor e_q
=
u_i u_q\,\ehat{i}\tensor e_q.
\]
Since
\[
u_i u_q\,\ehat{i}\tensor e_q
=
\left(u_i u_q-\frac17g_{iq}\right)\ehat{i}\tensor e_q
+
\frac17g_{iq}\ehat{i}\tensor e_q,
\]
we obtain
\[
R_{\uhat}\eta_1=\eta_4+\frac17\eta_2.
\]

Next,
\[
R_{\uhat}\eta_2
=
g_{iq}(\ehat{i}\uhat)\tensor e_q.
\]
Using
\[
\ehat{i}\uhat=-u_i\,1+\varphi_{ik\ell}u_k\ehat{\ell},
\]
and relabeling indices, this gives
\[
R_{\uhat}\eta_2
=
-u_q\,1\tensor e_q
+\varphi_{\ell qk}u_k\ehat{\ell}\tensor e_q
=
-\eta_1+\eta_3.
\]

For \(\eta_3\), we compute
\[
R_{\uhat}\eta_3
=
\varphi_{iqk}u_k(\ehat{i}\uhat)\tensor e_q.
\]
Again using
\[
\ehat{i}\uhat=-u_i\,1+\varphi_{i\ell m}u_\ell\ehat{m},
\]
the real part vanishes because \(\varphi_{iqk}u_i u_k=0\). The imaginary part is
\[
\varphi_{iqk}\varphi_{i\ell m}u_k u_\ell\,\ehat{m}\tensor e_q.
\]
The standard contraction identity gives
\begin{align*}
\varphi_{iqk}\varphi_{i\ell m}u_k u_\ell
&=
\bigl(g_{q\ell}g_{km}-g_{qm}g_{k\ell}+\psi_{q\ell km}\bigr)u_k u_\ell \\
&=
u_q u_m-g_{qm}.
\end{align*}
Using
\[
u_mu_q\ehat{m}\tensor e_q=\eta_4+\frac17\eta_2,
\]
we get
\[
R_{\uhat}\eta_3
= (u_mu_q - g_{mq})\ehat{m}\tensor e_q
=
-\frac67\eta_2+\eta_4.
\]

Finally,
\[
R_{\uhat}\eta_4
=
R_{\uhat}\left(u_i u_q\ehat{i}\tensor e_q-\frac17g_{iq}\ehat{i}\tensor e_q\right).
\]
The first term gives
\[
u_i u_q(\ehat{i}\uhat)\tensor e_q
=
(\uhat\uhat)\tensor u
=
-1\tensor u
=
-\eta_1,
\]
while the second term contributes
\[
-\frac17R_{\uhat}\eta_2
=
-\frac17(-\eta_1+\eta_3)
=
\frac17\eta_1-\frac17\eta_3.
\]
Therefore
\[
R_{\uhat}\eta_4
=
-\frac67\eta_1-\frac17\eta_3.
\]

Equations \eqref{eq:Ru-one-1}--\eqref{eq:Ru-one-4} give
\begin{equation}\label{eq:one-block-Ru-matrix}
[R_{\uhat}]_{\one^{\oplus4}}
=
\begin{pmatrix}
0 & -1 & 0 & -\frac67\\[2pt]
\frac17 & 0 & -\frac67 & 0\\[2pt]
0 & 1 & 0 & -\frac17\\[2pt]
1 & 0 & 1 & 0
\end{pmatrix}.
\end{equation}
Since
\[
H_{u,Q}=h-Q R_{\uhat},
\]
we obtain
\[
M_{\one}(Q)
=
[h]_{\one^{\oplus4}}-Q[R_{\uhat}]_{\one^{\oplus4}},
\]
which is exactly
\[
\begin{pmatrix}
-1 & Q & 6 & \frac67 Q\\[2pt]
-\frac17 Q & -7 & \frac67 Q & 0\\[2pt]
1 & -Q & 2 & \frac17 Q\\[2pt]
-Q & 0 & -Q & 0
\end{pmatrix}.
\]
This proves the lemma.
\end{proof}
\begin{corollary}\label{cor:one-block-quartic}
The \(\one^{\oplus 4}\)-block contributes octonionic right eigenvalues
\[
\lambda=\mu+Q\uhat
\]
of \(h\) precisely when \(\mu\) and $Q$ are real roots of the quartic equation
\begin{equation}\label{eq:one-block-quartic}
p_{\one}(\mu,Q)
=\mu^4
+6\mu^3
+(2Q^2-15)\mu^2
+(6Q^2-56)\mu
+Q^4+Q^2
=0.
\end{equation}
\end{corollary}

\begin{proof}
By Lemma~\ref{lem:one-block}, the restriction of \(H_{u,Q}\) to the \(\one^{\oplus4}\)-block is represented by the matrix \(M_{\one}(Q)\). Its characteristic polynomial is
\[
\det(M_{\one}(Q)-\mu I)
=
\mu^4
+6\mu^3
+(2Q^2-15)\mu^2
+(6Q^2-56)\mu
+Q^4+Q^2.
\]
Thus \(\mu\) is an ordinary real eigenvalue of \(H_{u,Q}\) on this block precisely when the displayed quartic vanishes. By Proposition~\ref{prop:right-eigenvalue-reduction}, this is equivalent to saying that
\[
\lambda=\mu+Q\uhat
\]
is an octonionic right eigenvalue of \(h\) arising from the \(\one^{\oplus4}\)-block.
\end{proof}

\begin{remark}
As a consistency check, setting \(Q=0\) in the quartic from Corollary~\ref{cor:one-block-quartic} gives
\[
p_{\one}(\mu,0)
=
\mu^4+6\mu^3-15\mu^2-56\mu
=
\mu(\mu+7)(\mu^2-\mu-8).
\]
Thus the \(\one^{\oplus4}\)-block recovers exactly the real eigenvalues of \(h\) coming from the \(G_2\)-summands represented inside this \(\SU(3)\)-isotypic component: the \(\mathbf{27}\)-summand contributes \(\mu=0\), the \(\mathbf 1\)-summand contributes \(\mu=-7\), and the two \(\mathbf 7\)-summands contribute the roots of
\[
\mu^2-\mu-8=0.
\]
This agrees with the real spectral computation of Theorem~\ref{thm:real-spectrum}.
\end{remark}
\subsection{The \(\six^{\oplus 4}\)-block}

We next analyze the \(\six^{\oplus 4}\)-isotypic component. Let
\[
J=J_u:u^\perp\to u^\perp,
\qquad
Jx=u\times x.
\]

With this convention, left and right multiplication by \(\uhat\) act on an octonion \(a\,1+\hat v\in\mathbb O\) by
\[
\uhat(a\,1+\hat v)
=
-g(u,v)\,1+a\uhat+\widehat{J_uv},
\]
whereas
\[
(a\,1+\hat v)\uhat
=
-g(u,v)\,1+a\uhat-\widehat{J_uv}.
\]

The four copies of \(\six\) arise from the two \(\seven\)-summands, the \(\fourteen\)-summand, and the \(\twentyseven\)-summand in the \(G_2\)-decomposition of \(W\). We use the following \(G_2\)-adapted embeddings of \(u^\perp\).

For \(x\in u^\perp\), define
\begin{equation}\label{eq:six-block-embeddings}
\begin{aligned}
\xi_0(x)
&:=1\otimes x,\\
\xi_7(x)
&:=
\left(
u_i x_q-u_q x_i+\psi_{iqk\ell}u_kx_\ell
\right)\ehat{i}\otimes e_q,\\
\xi_{14}(x)
&:=
\left(
u_i x_q-u_q x_i-\frac12\psi_{iqk\ell}u_kx_\ell
\right)\ehat{i}\otimes e_q,\\
\xi_{27}(x)
&:=
\left(
u_i x_q+u_q x_i
\right)\ehat{i}\otimes e_q.
\end{aligned}
\end{equation}
Thus \(\xi_0(x)\) lies in the real \(\seven\)-summand, \(\xi_7(x)\) lies in the \(\six\)-summand inside the imaginary \(\seven\), \(\xi_{14}(x)\) lies in the \(\six\)-summand inside \(\fourteen\), and \(\xi_{27}(x)\) lies in the \(\six\)-summand inside \(\twentyseven\).

\begin{remark} \label{rem:six-parametrization}
The \(\seven\)-summand in \(\Lambda^2V\) is being parametrized differently here than in Section~\ref{sec:real-spectrum}. In the real-spectrum computation we used the usual parametrization \(v\mapsto v\lrcorner\varphi\). In the present \(\six\)-block, the natural input is \(x\in u^\perp\), so we start from \(u\wedge x\) and take its \(\Lambda^2_7\)-projection. Up to the normalization used above, this is the same as \((u\times x)\lrcorner\varphi\). Similarly, \(\xi_{14}(x)\) is the \(\Lambda^2_{14}\)-projection of \(u\wedge x\). The chosen normalizations are convenient for the block calculation rather than orthonormal.
\end{remark}

A general element of the \(\six^{\oplus4}\)-block may therefore be written as
\[
\what
=
\xi_0(x_0)+\xi_7(x_7)+\xi_{14}(x_{14})+\xi_{27}(x_{27}),
\qquad
x_0,x_7,x_{14},x_{27}\in u^\perp.
\]

Since the stabilizer of \(u\) is \(\SU(3)\), each multiplicity space carries the complex structure induced by \(J\). It is convenient to complexify and identify \(J\) with multiplication by \(i\). Equivalently, every occurrence of \(i x\) below may be read in real terms as \(Jx\).

\begin{lemma}\label{lem:six-block}
With respect to the ordered basis
\[
(\xi_0,\xi_7,\xi_{14},\xi_{27})
\]
of the \(\six^{\oplus4}\)-multiplicity space, the complexified action of
\[
H_{u,Q}=h-Q R_{\uhat}
\]
is represented by the matrix
\begin{equation}\label{eq:six-block-matrix}
M_{\six}(Q)
=
\begin{pmatrix}
-1 & Q+6i & Q & Q\\[2pt]
-\frac{Q}{6}-i & 2-\frac{iQ}{2} & \frac{iQ}{2} & -\frac{iQ}{6}\\[2pt]
-\frac{Q}{3} & iQ & -4 & -\frac{iQ}{3}\\[2pt]
-\frac{Q}{2} & -\frac{iQ}{2} & -\frac{iQ}{2} & \frac{iQ}{2}
\end{pmatrix}.
\end{equation}
Equivalently, if
\[
\what
=
\xi_0(x_0)+\xi_7(x_7)+\xi_{14}(x_{14})+\xi_{27}(x_{27}),
\]
then
\[
H_{u,Q}(\what)
=
\xi_0(x'_0)+\xi_7(x'_7)+\xi_{14}(x'_{14})+\xi_{27}(x'_{27}),
\]
where, in real notation,
\[
\begin{aligned}
x'_0
&=
-x_0+6Jx_7+Qx_7+Qx_{14}+Qx_{27},\\
x'_7
&=
-\frac{Q}{6}x_0-Jx_0
+2x_7-\frac{Q}{2}Jx_7
+\frac{Q}{2}Jx_{14}
-\frac{Q}{6}Jx_{27},\\
x'_{14}
&=
-\frac{Q}{3}x_0
+QJx_7
-4x_{14}
-\frac{Q}{3}Jx_{27},\\
x'_{27}
&=
-\frac{Q}{2}x_0-\frac{Q}{2}Jx_7-\frac{Q}{2}Jx_{14}+\frac{Q}{2}Jx_{27}.
\end{aligned}
\]
\end{lemma}

\begin{proof}
The \(h\)-part is already known from the real spectral calculation in Section~\ref{sec:real-spectrum}. Namely, \(h\) acts on the two \(\seven\)-summands by the same \(2\times2\) block as before, while it acts by \(-4\) on the \(\fourteen\)-summand and by \(0\) on the \(\twentyseven\)-summand.

We now translate the \(\seven\oplus\seven\)-block into the present parametrization. In alignment with the real spectral calculation, the two copies of \(\seven\) are parametrized by
\[
(\xi_0,\xi_7),
\qquad
\xi_0=x_0\in V,
\qquad
\xi_7=J_u(x_7) \in V.
\]
Here the second equality is the convention of Remark~\ref{rem:six-parametrization}: the \(\Lambda^2_7\)-component is parametrized by the vector \(J_u(x_7)\), rather than directly by \(x_7\).

By Lemma~\ref{lem:h-block-computation}, \(h\) acts on the pair \((\xi_0,\xi_7)\) by
\[
\begin{pmatrix}
-1 & 6\\
1 & 2
\end{pmatrix}.
\]
Therefore
\[
(\xi_0,\xi_7)
\longmapsto
\left(
-\xi_0+6J_u x_7,\;
\xi_0\lrcorner\varphi+2(J_u x_7)\lrcorner\varphi
\right).
\]
To rewrite the second component in the original \(x_7\)-coordinate, note that
\[
\xi_0\lrcorner\varphi+2J_u x_7\lrcorner\varphi
=
(J_u(-J_u\xi_0+2x_7))\lrcorner\varphi.
\]
Since \(\xi_0=x_0\), the induced action on the multiplicity variables \((x_0,x_7)\) is
\[
(x_0,x_7)
\longmapsto
\left(
-x_0+6J_u x_7,\;
-J_u x_0+2x_7
\right).
\]

The remaining \(G_2\)-components are immediate:
\[
x_{14}\longmapsto -4x_{14},
\qquad
x_{27}\longmapsto 0.
\]
After complexifying \(u^\perp\) and replacing \(J_u\) by multiplication by \(i\), we obtain the \(Q=0\) part of the above matrix:
\begin{align} \label{eq:h6matrix}
    \begin{pmatrix}
-1 & 6i & 0 & 0\\
-i & 2 & 0 & 0\\
0 & 0 & -4 & 0\\
0 & 0 & 0 & 0
\end{pmatrix}.
\end{align}

It remains to compute the contribution of \(-Q R_{\uhat}\). Applying \(R_{\uhat}\) to the four embeddings and projecting back to the ordered basis
\[
(\xi_0,\xi_7,\xi_{14},\xi_{27}),
\]
one obtains the following transformation on the multiplicity variables:
\begin{equation}\label{eq:six-block-Ru-variables}
\begin{pmatrix}
x_0\\ x_7\\ x_{14}\\ x_{27}
\end{pmatrix}
\longmapsto
\begin{pmatrix}
-x_7-x_{14}-x_{27}\\[2pt]
\frac16x_0+\frac12Jx_7-\frac12Jx_{14}+\frac16Jx_{27}\\[2pt]
\frac13x_0-Jx_7+\frac13Jx_{27}\\[2pt]
\frac12x_0 +\frac12Jx_7 +\frac12Jx_{14} - \frac12Jx_{27} 
\end{pmatrix}.
\end{equation}
Equivalently, after complexifying \(u^\perp\) and replacing \(J\) by multiplication by \(i\), the \(R_{\uhat}\)-matrix is
\begin{equation}\label{eq:six-block-Ru-matrix}
[R_{\uhat}]_{\six^{\oplus4}}
=
\begin{pmatrix}
0 & -1 & -1 & -1\\[2pt]
\frac16 & \frac{i}{2} & -\frac{i}{2} & \frac{i}{6}\\[2pt]
\frac13 & -i & 0 & \frac{i}{3}\\[2pt]
\frac12 & \frac{i}{2} & \frac{i}{2} & -\frac{i}{2}
\end{pmatrix}.
\end{equation}
Since \(H_{u,Q}=h-Q R_{\uhat}\), and the \(h\)-part in this parametrization is given by \eqref{eq:h6matrix}, we obtain the displayed matrix \(M_{\six}(Q)\).
\end{proof}

\begin{corollary}\label{cor:six-block-no-nonreal}
The \(\six^{\oplus4}\)-block contributes no octonionic right eigenvalues of \(h\) with nonzero imaginary part in the slice
\[
\R\oplus \R\uhat .
\]
Equivalently, for real \(Q\neq 0\), the operator \(H_{u,Q}\) has no real eigenvalues on the \(\six^{\oplus4}\)-block.
\end{corollary}

\begin{proof}
By Lemma~\ref{lem:six-block}, after complexifying the multiplicity space and replacing \(J_u\) by multiplication by \(i\), the \(\six^{\oplus4}\)-block of \(H_{u,Q}\) is represented by \(M_{\six}(Q)\). Its characteristic polynomial is
\begin{equation}\label{eq:six-block-characteristic-polynomial}
\begin{aligned}
p_{\six}(\mu,Q)
&=
\mu^4+3\mu^3+(2Q^2-12)\mu^2+(3Q^2-32)\mu
+Q^4+4Q^2  \\
&\qquad
+iQ\left(3Q^2+3\mu^2+16\mu+16\right).
\end{aligned}
\end{equation}
A real eigenvalue of \(H_{u,Q}\) corresponds to a real number \(\mu\) satisfying
\[
p_{\six}(\mu,Q)=0.
\]
Thus both the real and imaginary parts of \(p_{\six}(\mu,Q)\) must vanish. Strictly speaking, \(p_{\six}\) is the determinant of the complex
multiplicity-space block corresponding to the \(J=i\) eigenspace after
complexification. The conjugate \(J=-i\) block has conjugate determinant.
For real \(\mu,Q\), singularity of the real \(6^{\oplus4}\)-block is therefore
equivalent to the vanishing of \(p_{\six}\), or equivalently to simultaneous
vanishing of its real and imaginary parts.

Since \(Q\neq0\), vanishing of the imaginary part gives

\[
3Q^2+3\mu^2+16\mu+16=0.
\]
Hence, solving for \(Q^2\), we get:
\begin{equation}\label{eq:six-Q2}
Q^2=-\frac13(\mu+4)(3\mu+4)
\end{equation}

Substituting \eqref{eq:six-Q2} into the real part of \eqref{eq:six-block-characteristic-polynomial}
\[
\operatorname{Re} p_6(\mu,Q)
=
-\frac{16}{9}(\mu+4)(2\mu-1).
\]
However, from \eqref{eq:six-Q2}, \(\mu=-4\) gives \(Q^2=0\), while, \(\mu=\frac12\) gives \(Q^2<0\).

Thus \(H_{u,Q}\) has no real eigenvalues on the \(\six^{\oplus4}\)-block for \(Q\neq0\). By Proposition~\ref{prop:right-eigenvalue-reduction}, this means that the \(\six^{\oplus4}\)-block contributes no octonionic right eigenvalues
\[
\lambda=\mu+Q\uhat
\]
of \(h\) with \(Q\neq0\).
\end{proof}

\subsection{The \(\eight^{\oplus 2}\)-block}

We next analyze the \(\eight^{\oplus 2}\)-isotypic component. Recall from Lemma~\ref{lem:8-to-8} that the \(\eight\)-summand inside
\[
\twentyseven\cong S^2_0V
\]
is naturally identified with the \(\eight\)-summand inside
\[
\fourteen\cong \Lambda^2_{14}V
\]
by the map
\[
\alpha\longmapsto \omega_{\alpha},
\qquad
\omega_{\alpha}(v,w)=\alpha (J_uv,w).
\]
Here \(\alpha\) is a traceless symmetric tensor on \(u^\perp\), extended by zero in the \(u\)-direction, and \(\omega_{\alpha}\) is regarded as a skew tensor in \(\Lambda^2_{14}V\).

Thus a general element of the \(\eight^{\oplus 2}\)-block may be written as
\[
\what=\alpha+\omega_{\beta},
\]
where
\[
\alpha,\beta \in \eight\subset S^2_0(u^\perp).
\]
Here \(\alpha\) parametrizes the \(\eight\)-summand inside \(\twentyseven\), while \(\beta\) parametrizes the \(\eight\)-summand inside \(\fourteen\).

\begin{lemma}\label{lem:eight-block}
With respect to the ordered basis
\[
(\alpha,\beta)
\]
of the \(\eight^{\oplus2}\)-multiplicity space, the operator
\[
H_{u,Q}=h-Q R_{\uhat}
\]
is represented by
\begin{equation}\label{eq:eight-block-matrix}
M_{\eight}(Q)
=
\begin{pmatrix}
0 & Q\\
-Q & -4
\end{pmatrix}.
\end{equation}
Equivalently,
\[
H_{u,Q}(\alpha+\omega_{\beta})
=
(Q \beta)-(Q\omega_{\alpha}+4\omega_{\beta}).
\]
\end{lemma}

\begin{proof}
The \(h\)-part is known from the real spectral calculation. Since
\[
\alpha\in \twentyseven,
\qquad
\omega_{\beta}\in \fourteen,
\]
we have
\[
h(\alpha)=0,
\qquad
h(\omega_{\beta})=-4\omega_{\beta}.
\]
Thus
\[
h(\alpha+\omega_{\beta})=-4\omega_{\beta}.
\]

It remains to compute the effect of \(R_{\uhat}\). We claim that
\[
R_{\uhat}(\alpha)=\omega_{\alpha},
\qquad
R_{\uhat}(\omega_{\beta})=-\beta.
\]
To see this, write \(\alpha=\alpha_{iq}\,\ehat{i}\otimes e_q\), with
\[
\alpha_{iq}=\alpha_{qi},
\qquad
u_i \alpha_{iq}=0,
\qquad
\alpha_{pp}=0,
\]
and with \(\alpha\) being \(J_u\)-invariant. Right multiplication by \(\uhat\) gives
\[
R_{\uhat}(\alpha)
=
\alpha_{iq}(\ehat{i}\uhat)\otimes e_q
=\alpha_{iq}(-u_i + \varphi_{ik\ell} u_k \ehat{\ell})\otimes e_q.
\]
Since \(\alpha\) has no \(u\)-component, the scalar part vanishes, and the imaginary part becomes:
\[
(\alpha_{iq}\varphi_{ik\ell} u_k) \ehat{\ell}\otimes e_q
=\alpha_{iq}(J_u)_{i\ell}\ehat{\ell}\otimes e_q=\omega_{\alpha},
\]
where we have used the convention \(J_u (v)=(J_u)_{ij}v_j e_i=(u\times v)_i e_i\).

Under the identification
\[
\omega_\alpha(v,w)=\alpha(J_uv,w),
\]
the action of \(R_{\uhat}\) sends \(\alpha\) to \(\omega_\alpha\). Applying the same operation to
\(\omega_\beta(v,w)=\beta(J_uv,w)\) applies \(J_u\) once more to the first argument, and hence gives
\[
R_{\uhat}(\omega_\beta)(v,w)=\beta(J_u^2v,w)=-\beta(v,w).
\]
Thus
\[
R_{\uhat}(\alpha)=\omega_\alpha,\qquad
R_{\uhat}(\omega_\beta)=-\beta.
\]

Therefore
\[
H_{u,Q}(\alpha+\omega_{\beta})
=
h(\alpha+\omega_{\beta})-Q R_{\uhat}(\alpha+\omega_{\beta})
=
-4\omega_{\beta}-Q(\omega_{\alpha}-\beta).
\]
Hence
\[
H_{u,Q}(\alpha+\omega_{\beta})
=
(Q\beta)-(Q\omega_{\alpha}+4\omega_{\beta}),
\]
which gives the matrix
\[
M_{\eight}(Q)
=
\begin{pmatrix}
0 & Q\\
-Q & -4
\end{pmatrix}.
\]
\end{proof}

\begin{corollary}\label{cor:eight-block-quadratic}
The \(\eight^{\oplus2}\)-block contributes octonionic right eigenvalues
\[
\lambda=\mu+Q\uhat
\]
of \(h\) precisely when \(Q, \mu\in\R\) satisfy
\[
Q^2+\mu^2+4\mu=0.
\]
Equivalently,
\[
(\mu+2)^2+Q^2=4.
\]
\end{corollary}

\begin{proof}
By Lemma~\ref{lem:eight-block}, the restriction of \(H_{u,Q}\) to the \(\eight^{\oplus2}\)-block is represented by
\[
M_{\eight}(Q)=
\begin{pmatrix}
0 & Q\\
-Q & -4
\end{pmatrix}.
\]
Thus \(\mu\) is an ordinary real eigenvalue of \(H_{u,Q}\) on this block precisely when
\[
\det(M_{\eight}(Q)-\mu I)=0.
\]
We compute
\[
\det
\begin{pmatrix}
-\mu & Q\\
-Q & -4-\mu
\end{pmatrix}
=
\mu(\mu+4)+Q^2
=
\mu^2+4\mu+Q^2.
\]
Therefore the block contributes octonionic right eigenvalues
\[
\lambda=\mu+Q\uhat
\]
of \(h\) precisely when
\[
Q^2+\mu^2+4\mu=0,
\]
by Proposition~\ref{prop:right-eigenvalue-reduction}.
\end{proof}

\section{The octonionic right spectrum}\label{sec:main-right-spectrum}

We now combine the block calculations from the previous section into a single description of the octonionic right-eigenvalue problem. Recall that for a fixed unit vector \(u\in V\), with corresponding \(\uhat\in \Vhat\), an octonionic right eigenvalue in the slice
\[
\C_{\uhat}=\R\oplus \R\uhat
\]
has the form
\[
\lambda=\mu+Q\uhat,
\qquad
\mu,Q\in \R,
\]
and corresponds to a real eigenvalue \(\mu\) of the operator
\[
H_{u,Q}=h-QR_{\uhat}.
\]

\begin{theorem}\label{thm:right-spectrum-main}
Fix a unit vector \(u\in V\) and consider the slice \(\C_{\uhat}=\R\oplus \R\uhat\). Then the octonionic right eigenvalues
\[
\lambda=\mu+Q\uhat
\]
of \(h\) are precisely those for which \((\mu,Q)\in \R^2\) lies on one of the following two loci:
\begin{enumerate}
    \item the quartic curve
    \begin{equation}\label{eq:right-spectrum-quartic}
    \mu^4
    +6\mu^3
    +(2Q^2-15)\mu^2
    +(6Q^2-56)\mu
    +Q^4+Q^2
    =0,
    \end{equation}
    arising from the \(\one^{\oplus4}\)-block;
    
    \item the circle
    \begin{equation}\label{eq:right-spectrum-circle}
    (\mu+2)^2+Q^2=4,
    \end{equation}

    arising from the \(\eight^{\oplus2}\)-block.
\end{enumerate}
The \(\six^{\oplus4}\)- and \(\twelve\)-blocks contribute no right eigenvalues with \(Q\neq 0\).

Equivalently, the non-real octonionic right eigenvalues of \(h\) in the slice \(\C_{\uhat}\) are exactly those
\[
\lambda=\mu+Q\uhat
\]
with \(Q\neq 0\) for which \((\mu,Q)\) lies on the quartic \eqref{eq:right-spectrum-quartic} or the circle \eqref{eq:right-spectrum-circle}.
\end{theorem}

\begin{proof}
By Proposition~\ref{prop:right-eigenvalue-reduction}, right eigenvalues
\[
\lambda=\mu+Q\uhat
\]
of \(h\) correspond to real eigenvalues \(\mu\) of \(H_{u,Q}\). By the \(\SU(3)\)-decomposition
\[
W\cong \one^{\oplus4}\oplus \six^{\oplus4}\oplus \eight^{\oplus2}\oplus \twelve,
\]
the operator \(H_{u,Q}\) splits into four independent blocks.

The \(\one^{\oplus4}\)-block contributes exactly the pairs \((\mu,Q)\) satisfying the quartic equation \eqref{eq:right-spectrum-quartic}, by Corollary~\ref{cor:one-block-quartic}. The \(\eight^{\oplus2}\)-block contributes exactly the pairs \((\mu,Q)\) satisfying the quadratic equation \eqref{eq:right-spectrum-circle}, by Corollary~\ref{cor:eight-block-quadratic}. By Corollary~\ref{cor:six-block-no-nonreal}, the \(\six^{\oplus4}\)-block contributes no real eigenvalues of \(H_{u,Q}\) for \(Q\neq0\), and hence no non-real right eigenvalues of \(h\). Finally, by Lemma~\ref{lem:twelve-block}, the \(\twelve\)-block also contributes no real eigenvalues of \(H_{u,Q}\) for \(Q\neq0\).

Thus the full right spectrum in the slice \(\C_{\uhat}\) is exactly the union of the quartic locus \eqref{eq:right-spectrum-quartic} and the circle \eqref{eq:right-spectrum-circle}.  Restricting to \(Q\neq0\) yields the non-real right spectrum. The \(\six\) and \(\twelve\) blocks only contribute real eigenvalues that are already included in the quartic and circle loci.
\end{proof}

\begin{figure}[ht]
\centering
\begin{tikzpicture}
\begin{axis}[
    width=1.2\textwidth,
    height=0.8\textwidth,
    axis equal image,
    xmin=-7.6, xmax=3.8,
    ymin=-3.3, ymax=3.3,
    axis lines=middle,
    xlabel={\(\mu\)},
    ylabel={\(Q\)},
    xlabel style={anchor=west},
    ylabel style={anchor=south},
    grid=both,
    grid style={gray!15},
    major grid style={gray!25},
    tick style={black},
    legend style={
        draw=none,
        fill=white,
        fill opacity=0.85,
        text opacity=1,
        at={(0.03,0.97)},
        anchor=north west
    },
    samples=250,
    domain=-7:3.3723,
    xtick={-7,-4,-2.4,-2,0,3.3},
    xticklabels={
        \(-7\),
        \(-4\),
        \(\mu_-\),
        \(-2\),
        \(0\),
        \(\mu_+\)
    },
    ytick={-3,-2,-1,0,1,2,3},
    legend style={
    draw=none,
    fill=white,
    fill opacity=0.85,
    text opacity=1,
    at={(0.03,0.97)},
    anchor=north west
    },
    legend image post style={line width=1.2pt},
]

% Quartic locus:
% Let Y = Q^2. Then
% Y^2 + (2 mu^2 + 6 mu + 1)Y
%      + mu^4 + 6 mu^3 - 15 mu^2 - 56 mu = 0.
% The relevant branch is
% Y = (-(2 mu^2+6mu+1) + sqrt(100mu^2+236mu+1))/2.

% Endpoints
\pgfmathsetmacro{\muminus}{(1-sqrt(33))/2}
\pgfmathsetmacro{\muplus}{(1+sqrt(33))/2}
\pgfmathsetmacro{\mudleft}{(-59-24*sqrt(6))/50}

% Discriminant and Q^2 branches
\pgfmathdeclarefunction{Disc}{1}{%
  \pgfmathparse{100*(#1)^2+236*(#1)+1}%
}

\pgfmathdeclarefunction{Yplus}{1}{%
  \pgfmathparse{max(0, (-(2*(#1)^2+6*(#1)+1) + sqrt(max(0,Disc(#1))))/2)}%
}

\pgfmathdeclarefunction{Yminus}{1}{%
  \pgfmathparse{max(0, (-(2*(#1)^2+6*(#1)+1) - sqrt(max(0,Disc(#1))))/2)}%
}

% Left large branch: Yplus
\addplot[
    red,
    very thick,
    domain=-7:\mudleft,
    samples=800,
    forget plot
]
({x},{sqrt(Yplus(x))});

\addplot[
    red,
    very thick,
    domain=-7:\mudleft,
    samples=800,
    forget plot
]
({x},{-sqrt(Yplus(x))});

% Left small closing branch: Yminus
\addplot[
    red,
    very thick,
    domain=\muminus:\mudleft,
    samples=400,
    forget plot
]
({x},{sqrt(Yminus(x))});

\addplot[
    red,
    very thick,
    domain=\muminus:\mudleft,
    samples=400,
    forget plot
]
({x},{-sqrt(Yminus(x))});

% Right branch: Yplus
\addplot[
    red,
    very thick,
    domain=0:\muplus,
    samples=800,
    forget plot
]
({x},{sqrt(Yplus(x))});

\addplot[
    red,
    very thick,
    domain=0:\muplus,
    samples=800,
    forget plot
]
({x},{-sqrt(Yplus(x))});

% Circle: Q^2 + mu^2 + 4mu = 0, i.e. (mu+2)^2 + Q^2 = 4.

\addplot[
    blue,
    very thick,
    domain=0:360,
    samples=300,
    forget plot
]
({-2+2*cos(x)},{2*sin(x)});

% Mark x=-7.
\addplot[
    only marks,
    mark=*,
    mark size=1.6pt,
    black,
    forget plot
]
coordinates {(-7,0)};

% Mark x=-4.
\addplot[
    only marks,
    mark=*,
    mark size=1.6pt,
    black,
    forget plot
]
coordinates {(-4,0)};

% Mark x=mu_-.
\addplot[
    only marks,
    mark=*,
    mark size=1.6pt,
    black,
    forget plot
]
coordinates {(\muminus,0)};

% Mark x=mu_+.
\addplot[
    only marks,
    mark=*,
    mark size=1.6pt,
    black,
    forget plot
]
coordinates {(\muplus,0)};

% Mark x=0.
\addplot[
    only marks,
    mark=*,
    mark size=1.6pt,
    black,
    forget plot
]
coordinates {(0,0)};

% Mark the circle center.
\addplot[
    only marks,
    mark=*,
    mark size=1.6pt,
    black,
    forget plot
]
coordinates {(-2,0)};

\addlegendimage{red, very thick}
\addlegendentry{\(\one^{\oplus4}\)-block}

\addlegendimage{blue, very thick}
\addlegendentry{\(\eight^{\oplus2}\)-block}

\node[below] at (axis cs:0.1,0) {\(0\)};
\end{axis}
\end{tikzpicture}
\caption{The slice spectral loci in the \((\mu,Q)\)-plane. The red quartic is the \(\one^{\oplus4}\)-block locus, while the blue circle is the \(\eight^{\oplus2}\)-block locus. Here \(\mu_\pm=\frac{1\pm\sqrt{33}}2\). }
\label{fig:loci}
\end{figure}
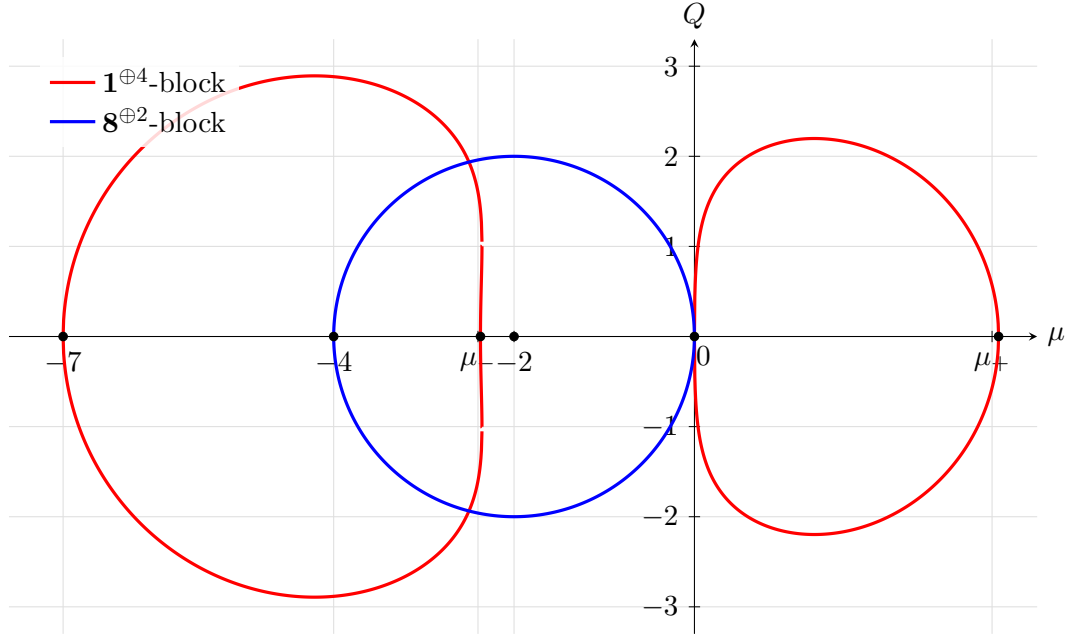

\begin{remark}\label{rem:loci-geometry}
Figure~\ref{fig:loci} illustrates the two families of non-real right eigenvalues. The quartic locus \eqref{eq:right-spectrum-quartic} consists of two disjoint branches, while the \(\eight^{\oplus2}\)-block contributes the circle \eqref{eq:right-spectrum-circle}. These two loci intersect at exactly three points:
\[
(\mu,Q)=(0,0),
\qquad
\left(-\frac52,\frac{\sqrt{15}}{2}\right),
\qquad
\left(-\frac52,-\frac{\sqrt{15}}{2}\right).
\]
Indeed, eliminating \(Q^2\) using \(Q^2=-\mu^2-4\mu\) from the circle equation and substituting into the quartic yields
\[
\mu(2\mu+5)=0.
\]
Since points on the circle must satisfy \(Q^2=-\mu^2-4\mu\ge 0\), this gives precisely
\[
(\mu,Q)=(0,0),\qquad
\left(-\frac52,\pm\frac{\sqrt{15}}2\right).
\]
In particular, the two nonzero intersection points correspond to the non-real right eigenvalues
\[
\lambda=-\frac52 \pm \frac{\sqrt{15}}{2}\,\uhat.
\]
\end{remark}

\begin{corollary}\label{cor:full-right-spectrum}
The full octonionic right spectrum of \(h\) is obtained from any fixed slice computation by \(G_2\)-symmetry. More precisely,
\begin{equation}\label{eq:full-right-spectrum-set}
\Spec_{\mathrm{right}}(h)
=
\left\{
\mu+Q\uhat\in \O:
\uhat\in \Im\O,\ |\uhat|=1,\ 
(\mu,Q)\in \mathcal L
\right\},
\end{equation}
where \(\mathcal L\subset \R^2\) is the union of the quartic locus
\[
\mu^4
+6\mu^3
+(2Q^2-15)\mu^2
+(6Q^2-56)\mu
+Q^4+Q^2
=0
\]
and the circle
\[
Q^2+\mu^2+4\mu=0.
\]
Equivalently, the non-real right spectrum is obtained by rotating the non-real points of these two planar loci through the \(G_2\)-orbit of a chosen unit imaginary direction.
\end{corollary}

\begin{proof}
Fix a unit imaginary octonion \(\uhat\in \Im\O\). The slice computation for
\[
\C_{\uhat}=\R\oplus\R\uhat
\]
shows that right eigenvalues of \(h\) in this slice are exactly the points
\[
\lambda=\mu+Q\uhat
\]
for which \((\mu,Q)\in\mathcal L\). The two equations defining \(\mathcal L\) arise from the \(\SU(3)\)-block decomposition associated to the stabilizer of this particular \(\uhat\).

Now let \(\gamma\in G_2\). Since \(h\) is \(G_2\)-equivariant, if
\[
h(\what)=\what\lambda,
\]
then
\[
h(\gamma\cdot\what)
=
\gamma\cdot h(\what)
=
\gamma(\what\lambda)
=
(\gamma\cdot\what)\,\gamma(\lambda).
\]
Thus \(\gamma(\lambda)\) is again a right eigenvalue of \(h\). Since \(G_2\) acts transitively on the unit sphere
\[
S^6\subset \Im\O,
\]
the right spectrum in every slice \(\R\oplus\R\uhat'\) is obtained from the fixed slice by some \(\gamma\in G_2\) with \(\gamma(\uhat)=\uhat'\).

Finally, every non-real octonion lies in a unique slice
\[
\R\oplus \R\uhat,
\qquad
\uhat=\frac{\Im(\lambda)}{|\Im(\lambda)|}.
\]
Therefore the union over \(|\uhat|=1\) of the fixed slice locus gives the full non-real right spectrum. The real points \(Q=0\) are included in the same equations and agree with the real spectrum computed in Theorem~\ref{thm:real-spectrum}.
\end{proof}

\begin{remark}
For each fixed choice of \(\uhat\), the quartic locus comes from the \(\one^{\oplus4}\)-block and the circle from the \(\eight^{\oplus2}\)-block in the \(\SU(3)=\operatorname{Stab}_{G_2}(\uhat)\)-decomposition. These block decompositions vary with the slice, while the union of the resulting loci is \(G_2\)-invariant.
\end{remark}

\begin{remark}
    Geometrically, the full non-real spectrum is best viewed as a
multi-sheeted radial hypersurface in
\[
\mathbb O=\mathbb R\oplus \operatorname{Im}\mathbb O.
\]
Indeed, the \(G_2\)-invariance implies that the defining equations depend
only on
\[
\mu=\operatorname{Re}\lambda,
\qquad
r=|\operatorname{Im}\lambda|.
\]
Thus each non-real point \((\mu,r)\) of the slice spectral locus lifts to
the \(G_2\)-orbit
\[
\{\mu+r\uhat:|\uhat|=1\}\cong S^6.
\]
Equivalently,
\[
\operatorname{Spec}_{\mathrm{right}}(h)\setminus\mathbb R
=
\left\{
\mu+y:\ y\in\operatorname{Im}\mathbb O,\ y\neq0,\ 
(\mu,|y|)\in L
\right\},
\]
where \(L\) is the union of the quartic and circle loci in the
\((\mu,r)\)-half-plane.  Hence the spectrum projects to a ball in
\(\operatorname{Im}\mathbb O\), but with a finite number of spectral sheets
over each radius.  The number of sheets changes only at the fold radii of
the quartic component, at the cap \(r=2\) of the circle component, and at the
quartic--circle intersection \(r=\sqrt{15}/2\). The number of sheets can be seen in Figure \ref{fig:loci} as the number of values of \(
\mu\) that correspond to any given value of \(Q\). In particular, the ball has radius \(Q_{\func{max}}\), the maximal value of \(|Q|\) attained by the left quartic branch. The maximal number of sheets is then \(6\), which is attained when both branches of the quartic and the circle locus contribute two values of \(\mu\) each. At the real axis the
\(S^6\)-orbits collapse and the five real eigenvalues are recovered.
\end{remark}

The eigenspace dimensions vary around the displayed loci in Figure \ref{fig:loci}. For the \(\one^{\oplus4}\) and \(\eight^{\oplus2}\) blocks, a direct computation of the eigenvectors gives the following.

\begin{corollary}[Eigenspace dimensions]\label{cor:right-eigenspace-dimensions}
Fix a unit imaginary octonion \(\uhat\), and let
\[
\lambda=\mu+Q\uhat\in \R\oplus \R\uhat .
\]
All eigenspace dimensions below are real dimensions.

Let
\[
E_{\lambda}
=
\ker(h-R_{\lambda})
=
\ker(H_{u,Q}-\mu I).
\]
Then the following hold.

\begin{enumerate}
\item Suppose \(Q\neq0\), and suppose \((\mu,Q)\) lies on the quartic locus
\[
p_{\one}(\mu,Q)=0
\]
but not on the circle locus
\[
(\mu+2)^2+Q^2=4.
\]
Then
\[
\dim E_{\lambda}=1.
\]
More precisely, the eigenspace lies in the \(\one^{\oplus4}\)-block and is
spanned by
\[
v_1(\mu,Q)
=
x_1\eta_1+x_2\eta_2+x_3\eta_3+x_4\eta_4,
\]
where one may take
\[
\begin{aligned}
x_1&=6(Q^2-\mu^2-7\mu),\\
x_2&=-\frac{6Q}{7}(Q^2+\mu^2),\\
x_3&=-\bigl(Q^2(\mu+6)+\mu(\mu+1)(\mu+7)\bigr),\\
x_4&=Q\bigl(Q^2+(\mu+7)^2\bigr).
\end{aligned}
\]
Since \(Q\neq0\), the component \(x_4\) is nonzero.

\item Suppose \(Q\neq0\), and suppose \((\mu,Q)\) lies on the circle locus
\[
(\mu+2)^2+Q^2=4
\]
but not on the quartic locus. Then
\[
\dim E_{\lambda}=8.
\]
More precisely, the eigenspace lies in the \(\eight^{\oplus2}\)-block and is
\[
E_{\lambda}
=
\left\{
Q\alpha+\mu\,\omega_{\alpha}:\alpha\in \eight
\right\}.
\]

\item At the two non-real intersection points
\[
(\mu,Q)
=
\left(-\frac52,\pm\frac{\sqrt{15}}2\right),
\]
both the \(\one^{\oplus4}\)-block and the \(\eight^{\oplus2}\)-block contribute.
In this case
\[
\dim E_{\lambda}=1+8=9.
\]
More explicitly,
\[
E_{\lambda}
=
\operatorname{span}
\left\{
105\eta_1-10Q\eta_2-35\eta_3+28Q\eta_4
\right\}
\oplus
\left\{
-2Q\alpha+5\omega_{\alpha}:\alpha\in \eight
\right\}.
\]

\item For real eigenvalues, i.e. \(Q=0\), the eigenspaces recover the
ordinary real eigenspaces of \(h\):
\[
\dim E_{-7}=1,
\qquad
\dim E_{-4}=14,
\qquad
\dim E_0=27,
\]
and
\[
\dim E_{\frac{1-\sqrt{33}}2}
=
\dim E_{\frac{1+\sqrt{33}}2}
=
7.
\]
\end{enumerate}
\end{corollary}
\begin{proof}
For \(Q\neq0\), the \(\six^{\oplus4}\)- and \(\twelve\)-blocks contribute no real
eigenvalues of \(H_{u,Q}\), by Corollary~\ref{cor:six-block-no-nonreal}
and Lemma~\ref{lem:twelve-block}. Thus only the \(\one^{\oplus4}\)- and
\(\eight^{\oplus2}\)-blocks can contribute.

On the \(\one^{\oplus4}\)-block, the vector \(v_1(\mu,Q)\) satisfies
\[
(M_{\one}(Q)-\mu \id)v_1(\mu,Q)=0
\]
whenever \(p_{\one}(\mu,Q)=0\). Moreover, for \(Q\neq0\), a \(3\times3\) minor
of \(M_{\one}(Q)-\mu \id\) is nonzero; hence the kernel is one-dimensional on the
quartic locus.

On the \(\eight^{\oplus2}\)-block, the matrix is
\[
M_8(Q)=
\begin{pmatrix}
0&Q\\
-Q&-4
\end{pmatrix}.
\]
On the circle locus, the vector \((Q,\mu)\) spans the kernel of
\(M_8(Q)-\mu \id\). Since the corresponding \(SU(3)\)-type is the real
eight-dimensional adjoint representation, this gives the eigenspace
\[
\left\{
Q\alpha+\mu\omega_{\alpha}:\alpha\in \eight
\right\}.
\]
At the two non-real intersection points, these two contributions occur
simultaneously and lie in distinct \(SU(3)\)-isotypic blocks, so the
eigenspace is their direct sum.

Finally, when \(Q=0\), the right-eigenvalue equation reduces to the ordinary
real eigenvalue problem for \(h\), as given by Theorem~\ref{thm:real-spectrum}.
\end{proof}

\section{Concluding remarks}\label{sec:concluding-remarks}

The results of this paper show that the octonionic right spectrum of the canonical operator \(h\) is governed by representation theory of $\Gtwo$ and $\SU(3)$. The ordinary real spectrum is determined by the \(G_2\)-decomposition
\[
W\cong \one\oplus \seven\oplus \seven\oplus \fourteen\oplus \twentyseven,
\]
while the non-real right-eigenvalue equation
\[
h(\what)=\what\lambda
\]
is reduced, after choosing a slice \(\R\oplus\R\uhat\), to an \(\SU(3)=\operatorname{Stab}_{G_2}(\uhat)\)-equivariant problem.

In each such slice, the full non-real right spectrum is controlled by two explicit algebraic loci. The \(\one^{\oplus4}\)-block gives the quartic
\[
\mu^4
+6\mu^3
+(2Q^2-15)\mu^2
+(6Q^2-56)\mu
+Q^4+Q^2
=0,
\]
and the \(\eight^{\oplus2}\)-block gives the circle
\[
Q^2+\mu^2+4\mu=0.
\]
The \(\six^{\oplus4}\)- and \(\twelve\)-blocks contribute no non-real right eigenvalues. Thus the right spectrum is not only computable, but organized by a small number of canonical equations in the invariant quantities \(\mu=\operatorname{Re}\lambda\) and \(Q=|\operatorname{Im}\lambda|\).

These planar loci are slice models for a \(G_2\)-symmetric spectral geometry in \(\mathbb O\). They are independent of coordinates, bases, or a chosen multiplication table: choosing \(\uhat\) merely selects a representative complex slice, and the full right spectrum is obtained by rotating this slice under the transitive \(G_2\)-action on the unit sphere in \(\Im\mathbb O\). In this sense, the quartic and circle record the intrinsic symmetry breaking \(G_2\to\SU(3)\).

Equivalently, these two loci admit an intrinsic \(G_2\)-invariant trace--norm formulation. For \(\lambda\in\mathbb O\), let
\[
T(\lambda)=\lambda+\bar\lambda,
\qquad
N(\lambda)=\lambda\bar\lambda.
\]
Then, for \(\lambda=\mu+Q\uhat\), \( T(\lambda) = 2\mu\) and \(N(\lambda)=\mu^2+Q^2\). Hence, the quartic component can be written as
\begin{equation} \label{eq:quartic-trace-norm}
N(\lambda)^2+\bigl(3T(\lambda)+1\bigr)N(\lambda)
-4T(\lambda)^2-28T(\lambda)=0,
\end{equation}
while the circle component becomes
\begin{equation} \label{eq:circle-trace-norm}
N(\lambda)+2T(\lambda)=0.
\end{equation}
Thus, away from the real axis, the full right spectrum is the union of the non-real parts of two \(\Gtwo\)-invariant real algebraic hypersurfaces in \(\O\).

The operator \(h\) therefore gives a concrete example of genuinely octonionic spectral behavior. A natural Hermitian \(G_2\)-equivariant operator has continuous families of non-real right eigenvalues, so its right spectrum should be viewed as an intrinsic octonionic spectral object: the alignment locus where \(h\) agrees with right multiplication by an octonion, and hence a probe of how nonassociativity enters operator theory.

One way to interpret the right-spectrum computed here is as a symmetry-alignment problem. Suppose $U$ is a vector space over a field $F$.  Then,  \(F\)-linearity of an operator \(T\) on  \(U\) means equivariance with respect to the scalar action of \(F^*\) (the group of non-zero elements of $F$), while the comparison operator \(\lambda\id_U\) is invariant under the full group \(\GL_F(U)\). The eigenvalue equation \(Tv=\lambda v\) therefore tests when \(T\) agrees on a nonzero vector with a maximally symmetric scalar operator. In the octonionic problem studied here, the roles are shifted: the operator \(h\) is already highly symmetric, being \(G_2\)-equivariant, whereas right multiplication by a non-real octonion \(\lambda=\mu+Q\hat u\) selects a complex slice and has only the stabilizer symmetry \(\SU(3)\). Thus the spectral equation \(h(\hat w)=\hat w\lambda\) measures alignment between operators of different symmetry type, and the resulting block decomposition reflects the reduction \(G_2\to\SU(3)\).

The present approach also differs from earlier work on octonionic eigenvalue problems for small matrices. Much of the existing literature studies \(2\times2\) or \(3\times3\) octonionic Hermitian matrices, exceptional Jordan algebra models, and the special algebraic phenomena that arise from determinant-like identities and nonassociative matrix multiplication \cite{DrayJaneskyManogue2000NonRealEigenvalues,DrayManogue1998OctonionicEigenvalue,DrayManogue1999ExceptionalJordan,DrayManogueOkubo2002OrthonormalEigenbases,GillowWilesDray2010ImaginaryEigenvalues,DeLeoDucati2012OctonionicEigenvalue,Okubo1999Symmetric3x3Octonionic}. By contrast, the operator studied here is not chosen as a small octonionic matrix model, but is constructed intrinsically from the octonion multiplication tensor and the \(G_2\)-module \(W=\mathbb O\otimes_{\mathbb R}V\). The resulting spectral loci are therefore not artifacts of a particular matrix presentation: they arise from the branching of \(G_2\)-representations under the stabilizer \(\SU(3)\), and the quartic and circle equations record this symmetry reduction in an explicit algebraic form.

There are several natural directions for further work. First, the quartic and circle should admit an even more intrinsic description in terms of \(G_2\)- and \(\SU(3)\)-intertwining data. In the present paper they arise from explicit block matrices, but their coordinate independence suggests a formulation directly in the algebra of invariant tensors and multiplicity spaces. Such a formulation could clarify why precisely these two loci, and no others, survive in the right spectrum.

Second, the construction invites a manifold-level extension. A \(G_2\)-structure on a \(7\)-manifold defines a fiberwise octonion bundle with a nonassociative product \cite{Grigorian2017G2StructuresOctonionBundles}. In that setting, the operator studied here can be viewed as a pointwise algebraic model for natural operators on octonion-valued bundles. It would be interesting to investigate how the eigenvalue loci interact with torsion, covariant differentiation, and geometric operators on \(G_2\)-manifolds. The guiding analogy is complex geometry: after complexification, an almost complex structure decomposes the tangent bundle into its \((1,0)\)- and \((0,1)\)-eigenbundles, and this decomposition is central to the subject. As discussed in Section~\ref{sec:cx-analogy}, the operator \(h\) packages the slice-wise complex structures associated to the octonionic \(G_2\)-structure. Its spectral loci may therefore be viewed as candidates for octonionic decomposition data on \(G_2\)-manifolds.

A further direction is to organize octonionic spectral theory into a hierarchy of higher right-eigenvalue problems. The equation studied here,
\[
h(\what)=\what\lambda,
\]
is the first non-real level, comparing \(h\) with right multiplication by a single octonion. The ``zeroth'' level would be the ordinary real eigenvalue problem. We can note that the imaginary part of $\what\lambda$ is up to a multiple the commutator of $\what$ and $\lambda$.  Therefore, in a commutative setting, the first level problem collapses to the scalar level. One may next consider two-step eigenvalue equations such as
\[
h(\what)=(\what\lambda_1)\lambda_2.
\]
The difference between \((\what\lambda_1)\lambda_2\) and \(\what(\lambda_1\lambda_2)\) is precisely an associator term, and thus, in an associative setting, the second level collapses to the first level. As such, this hierarchy would separate ordinary scalar behavior, noncommutativity, and genuinely nonassociative effects.

\bibliographystyle{abbrvdin}
\bibliography{refs}
\end{document}